\documentclass[12pt,oneside,english]{amsart}
\usepackage[latin9]{inputenc}
\usepackage{geometry}
\usepackage{babel}
\usepackage{amstext}
\usepackage{amsthm}
\usepackage{amssymb}
\usepackage[unicode=true,pdfusetitle,
 bookmarks=true,bookmarksnumbered=false,bookmarksopen=false,
 breaklinks=false,pdfborder={0 0 1},backref=false,colorlinks=false]
 {hyperref}

\makeatletter
\numberwithin{equation}{section}
\numberwithin{figure}{section}
\theoremstyle{plain}
\newtheorem{thm}{\protect\theoremname}
  \theoremstyle{plain}
  \newtheorem{conjecture}[thm]{\protect\conjecturename}
  \theoremstyle{plain}
  \newtheorem{cor}[thm]{\protect\corollaryname}
  \theoremstyle{plain}
  \newtheorem{prop}[thm]{\protect\propositionname}
  \theoremstyle{plain}
  \newtheorem{lem}[thm]{\protect\lemmaname}

\usepackage{babel}
\newcommand{\F}{\mathbb{F}}
\newcommand{\Q}{\mathbb{Q}}
\newcommand{\Surj}{\textnormal{Sur}}
\newcommand{\Hom}{\textnormal{Hom}}
\newcommand{\disc}{\textnormal{disc}}
\newcommand{\Leg}[2]{\left(\frac{#1}{#2}\right)}

\usepackage{babel}
\providecommand{\conjecturename}{Conjecture}
  \providecommand{\corollaryname}{Corollary}
  \providecommand{\lemmaname}{Lemma}
  \providecommand{\propositionname}{Proposition}
\providecommand{\theoremname}{Theorem}

\makeatother

  \providecommand{\conjecturename}{Conjecture}
  \providecommand{\corollaryname}{Corollary}
  \providecommand{\lemmaname}{Lemma}
  \providecommand{\propositionname}{Proposition}
\providecommand{\theoremname}{Theorem}

\begin{document}

\title{The distribution of $H_{8}$-extensions of quadratic fields}

\author{Brandon Alberts, Jack Klys}
\begin{abstract}
We compute all the moments of a normalization of the function which
counts unramified $H_{8}$-extensions of quadratic fields, where $H_{8}$
is the quaternion group of order $8$, and show that the values of
this function determine a point mass distribution. Furthermore we
propose a similar modification to the non-abelian Cohen-Lenstra heuristics
for unramified $G$-extensions of quadratic fields for $G$ in a large
class of 2-groups, which we conjecture will give finite moments which
determine a distribution. Our method additionally can be used to determine
the asymptotics of the unnormalized counting function, which we also
do for unramified $H_{8}$-extensions. 
\end{abstract}

\maketitle

\section{Introduction}

The Cohen-Lenstra conjectures describe the distributions of $p$-class
groups of number fields in certain families \cite{cohenlenstra}.
By class field theory the class group is the Galois group of the maximal
abelian unramified extension. Thus there is a natural non-abelian
generalization of this question to describe the distribution of the
Galois groups of maximal unramified extensions.

Let $\mathcal{D}_{X}^{\pm}$ be the set of fundamental discriminants
$0<\pm d<X$. For any function $f$ on quadratic number fields define
\[
\mathbf{E}^{\pm}\left(f\right)=\lim_{X\longrightarrow\infty}\frac{\sum_{K,\disc K\in\mathcal{D}_{X}^{\pm}}f\left(K\right)}{\sum_{K,\disc K\in\mathcal{D}_{X}^{\pm}}1}.
\]
The Cohen-Lenstra conjecture for quadratic fields is equivalent to:
for all finite abelian $p$-groups $A$ 
\[
\mathbf{E}^{\pm}\left(\left|\mathrm{Sur}\left(Cl_{K,p},A\right)\right|\right)=\frac{1}{\left|A\right|^{u}}
\]
where $u=0,1$ in imaginary and real cases respectively and $Cl_{K,p}$
is the $p$ part of the class group over $K$. Thus in the non-abelian
case it is natural to study, for any non-abelian group $G$, 
\[
\mathbf{E}^{\pm}\left(\left|\mathrm{Sur}\left(\mathrm{Gal}\left(K^{un}/K\right),G\right)\right|\right).
\]
Note that this is equivalent to determining the number of unramified
extensions $L/K$ with $\mathrm{Gal}\left(L/K\right)=G$.

We will refine this question further. Fix a finite group $G'$ and
a subgroup $G$. Let $f\left(K\right)$ be the number of unramified
extensions $L/K$ with $\mathrm{Gal}\left(L/K\right)=G$ and $\mathrm{Gal}\left(L/\mathbb{Q}\right)=G'$.
For now we will consider extensions unramified only at the finite
primes. We will call such $L/K$ an unramified $\left(G',G\right)$-extension.
Define $\mathbf{E}^{\pm}\left(G',G\right)=\mathbf{E}^{\pm}\left(f\left(K\right)\right)$.
In \cite{BhargavaNonabelian} Bhargava asked about the value of $\mathbf{E}^{\pm}\left(G',G\right)$
and proved several cases 
\begin{thm}[Bhargava]
For $n=3,4,5$ 
\begin{eqnarray*}
\mathbf{E}^{\pm}\left(S_{n}\times C_{2},S_{n}\right) & = & \infty\\
\mathbf{E}^{+}\left(S_{n},A_{n}\right) & = & \frac{1}{n!}\\
\mathbf{E}^{-}\left(S_{n},A_{n}\right) & = & \frac{1}{2\left(n-2\right)!}.
\end{eqnarray*}
\end{thm}
Wood conjectured an answer to this question and proved some results
in function fields which support it \cite{Wood} (we refer there for
the precise definitions). 
\begin{conjecture}[Wood]
\label{conj:-melanie} Suppose there is a unique conjugacy class
$c$ of order 2 elements of $G'$ which are not contained in $G$.
Then 
\[
E^{\pm}\left(G',G\right)=\frac{\left|H_{2}\left(G',c\right)\left[2\right]\right|}{\left|c\right|^{u}\mathrm{Aut}{}_{G'}\left(G\right)}
\]
where $u=0,1$ in the imaginary and real cases respectively. Otherwise
$E^{\pm}\left(G',G\right)=\infty$. 
\end{conjecture}
Alberts verified \cite{alberts} the case of this conjecture $E^{\pm}\left(H_{8}\rtimes C_{2},H_{8}\right)$
where $H_{8}\rtimes C_{2}$ is isomorphic to the quotient of $D_{4}\oplus C_{4}$
obtained by identifying their Frattini subgroups and is the unique
group that can occur as $\mathrm{Gal}\left(L/\mathbb{Q}\right)$.
Throughout the rest of the paper we will let $H_{8}\rtimes C_{2}$
denote this group.

In this paper we consider $E^{\pm}\left(H_{8}\rtimes C_{2},H_{8}\right)$
with the imporant modification that $f\left(K\right)$ is appropriately
normalized. We show all of the $k$th moments of this new function
are finite, and in fact determine the point mass distribution. 

We will now change the definition of $f\left(K\right)$ to count extensions
unramified everywhere, including at the infinite prime, and call the
corresponding extensions unramified everywhere $\left(G',G\right)$-extensions.
We prove our theorems in this case- the case of extensions unramified
at only finite primes is simpler and follows easily from our work
(see Proposition \ref{prop:The-number-of}).

Our main result is 
\begin{thm}
\label{thm:mainthm}Let $\left(G',G\right)=\left(H_{8}\rtimes C_{2},H_{8}\right)$.
For a quadratic field $K$ with odd discriminant $d$ let $f\left(d\right)$
be the number of unramified everywhere $\left(G',G\right)$-extensions
of $K$ and let $g\left(d\right)=3^{\omega\left(d\right)}$. Then
for all $k\in\mathbb{Z}_{\ge1}$ 
\[
\mathbf{E}^{-}\left(\left(f/g\right)^{k}\right)=\left(\frac{1}{32}\right)^{k}
\]
and 
\[
\mathbf{E}^{+}\left(\left(f/g\right)^{k}\right)=\left(\frac{1}{192}\right)^{k}
\]
Thus the function $f\left(d\right)/g\left(d\right)$ determines the
point mass distribution on $\mathbb{R}$. 
\end{thm}
We mean by this last statement that the sequence of measures 
\[
\mu_{n}\left(U\right)=\frac{1}{\left|\mathcal{D}_{n}^{\pm}\right|}\left|\left\{ f\left(d\right)/g\left(d\right)\in U\mid d<n\right\} \right|
\]
on $\mathbb{R}$ converges to the point mass $\mu_{c}$ in distribution,
where $c=1/32$ and $1/192$ in the complex and real case respectively.

As a corollary we have the following result which ties back to the
question of the distribution of the Galois group of the maximal unramified
extension. 
\begin{cor}
\label{cor:The-density-of}The density of quadratic fields $K$ with
$\mathrm{Gal}\left(K^{un}/K\right)=H_{8}^{m}$ is equal to $0$ for
any positive $m\in\mathbb{Z}$. 
\end{cor}
Additionally, we prove an analogous unnormalized result which directly
generalizes the aforementioned results due to Alberts \cite{alberts}.
Let $H_{8}^{k}\rtimes_{\sigma}C_{2}$ denote the group where the action
of $\sigma$ on each coordinate gives $H_{8}\rtimes C_{2}$ according
to our definition. 
\begin{thm}
\label{thm:mainthm2}Let $k\in\mathbb{Z}_{\ge1}$, $G=H_{8}^{k}$,
and $G'=H_{8}^{k}\rtimes_{\sigma}C_{2}$. Define $\Surj_{\sigma}\left(\mathrm{Gal}\left(K^{un}/K\right),G\right)$
be the set of surjections which lift to a surjection $\mathrm{Gal}\left(K^{un}/\Q\right)\rightarrow G'$.
Then 
\[
\sum_{d\in\mathcal{D}_{X}^{-}<X}\left|\Surj_{\sigma}\left(\mathrm{Gal}\left(K^{un}/K\right),G\right)\right|=\left(\frac{1}{4}\right)^{k}\left(\sum_{d\in\mathcal{D}_{X}^{-}}3^{k\omega(d)}\right)+O\left(X\left(\log X\right)^{3^{k}-2+\epsilon}\right)
\]
and 
\[
\sum_{d\in\mathcal{D}_{X}^{+}}\left|\Surj_{\sigma}\left(\mathrm{Gal}\left(K^{un}/K\right),G\right)\right|=\left(\frac{1}{24}\right)^{k}\left(\sum_{d\in\mathcal{D}_{X}^{+}}3^{k\omega(d)}\right)+O\left(X\left(\log X\right)^{3^{k}-2+\epsilon}\right).
\]
\end{thm}
The proof of these results relies on a condition for the existence
of $H_{8}$-extensions and explicit construction thereof, due to Lemmermeyer
\cite{lemmermeyerh8} (see Proposition \ref{prop:The-number-of} below).
We use this obtain a formula for the number of such extensions of
any quadratic field. Then we build on the methods of Fouvry and Klüners
from \cite{fk1} to study the asymptotic growth of this expression.
The constant of the main term in this expression is then obtained
using combinatorial arguments with vector spaces over $\mathbb{F}_{2}$.
The proof is separated into cases depending on the sign and congruence
class of $d$ modulo 8. The cases are all qualatatively similar, so
we only present the case $d<0$ and $d\equiv1\mod4$ in the main body
of the paper. We include the computations necessary for the remaining
cases in an appendix for the sake of completeness.

For $1/3\le a$ let $\mathcal{D}_{X,n,m}^{\pm}$ be the set of discriminants
$d$ such that $\pm d>0$ and $d\equiv n\mod m$ and 
\[
S_{k}^{\pm}(X,n,m)=\sum_{d\in\mathcal{D}_{X,n,m}^{\pm}}\left(a^{\omega(d)}f(d)\right)^{k}.
\]
Both of the above theorems follow from the next result, the proof
of which constitues the bulk of this paper. 
\begin{thm}
\label{thm:mainthm3}Let $\left(G',G\right)=\left(H_{8}\rtimes C_{2},H_{8}\right)$.
For a quadratic field $K$ with discriminant $d$ let $f\left(d\right)$
be the number of unramified everywhere $\left(G',G\right)$-extensions
of $K$. Then for all $k\in\mathbb{Z}_{\ge1}$, $a\ge1/3$, and $(n,m)\in\{(4,8),(0,8)\}$
\begin{align*}
S_{k}^{-}(X,1,4) & =\frac{1}{2^{5k}}\left(\sum_{d\in\mathcal{D}_{X,1,4}^{-}}(3a)^{k\omega(d)}\right)+O\left(X\left(\log X\right)^{(3a)^{k}-a^{k}-1+\epsilon}\right)\\
S_{k}^{-}(X,n,m) & =\frac{1}{a^{k}2^{5k}}\left(\sum_{d\in\mathcal{D}_{X,n,m}^{-}}(3a)^{k\omega(d)}\right)+O\left(X\left(\log X\right)^{(3a)^{k}-a^{k}-1+\epsilon}\right)
\end{align*}
And for all $(n,m)\in\{(4,8),(0,8)\}$ 
\begin{align*}
S_{k}^{+}(X,1,4) & =\frac{1}{3^{k}2^{6k}}\left(\sum_{d\in\mathcal{D}_{X,1,4}^{+}}(3a)^{k\omega(d)}\right)+O\left(X\left(\log X\right)^{(3a)^{k}-a^{k}-1+\epsilon}\right)\\
S_{k}^{+}(X,n,m) & =\frac{1}{(3a)^{k}2^{6k}}\left(\sum_{d\in\mathcal{D}_{X,n,m}^{+}}(3a)^{k\omega(d)}\right)+O\left(X\left(\log X\right)^{(3a)^{k}-a^{k}-1+\epsilon}\right)
\end{align*}
\end{thm}
In fact Lemmermeyer exhibits conditions for the existence of a multitude
of unramified $\left(G',G\right)$-extensions \cite{lemmermeyer2groups}
all of which are $2$-groups with $\left[G':G\right]=2$. However
it should be noted that from that list only $H_{8}$ and $D_{4}$
can occur as the Galois group of an unramified extension of quadratic
field which is Galois over $\mathbb{Q}$.

However it seems plausible that there are other examples of the form
$\left(G',G\right)$ for which the corresponding extensions can be
counted using similar expression built out of Legendre symbols. In
such cases we would expect similar methods to work in computing finite
moments with the appropriate normalizations, though it is not clear
what distributions to expect. We make the following conjecture for
all unramified $(G',G)$-extensions, and which we make more specific
in the case of $\left(D_{4}\times C_{2},D_{4}\right)$ using the above
condition of Lemmermeyer combined with a heuristic which we detail
in Section \ref{sec:Future-Work}. 
\begin{conjecture}
\label{conj:For-a-quadratic} For a quadratic field $K$ with discriminant
$d$, and an admissible pair of $2$-groups $(G',G)$ as defined in
\cite{Wood}, let $f(d)$ be the number of  $(G',G)$-extensions of
$K$. Then there exists a number $0<a<1$ such that 
\[
\mathbf{E}\left(\left(a^{\omega(d)}f(d)\right)^{k}\right)<\infty
\]
and these moments determine a distribution.

In the case of $\left(G',G\right)=\left(D_{4}\times C_{2},D_{4}\right)$
we have $a=1/4$. 
\end{conjecture}
What the value of $a$ and the form of the distribution should be
is unclear at this time. The result of Theorem \ref{thm:mainthm}
does not seem to fit with the finite case of Conjecture \ref{conj:-melanie}.

In each case $a$ seems to be at least partially related to the genus
field $K^{gen}$ of $K$. For example in Theorem \ref{thm:mainthm}
where $a=1/3$ it can be seen that $g\left(d\right)=3^{\omega\left(d\right)}/6$
is the number of subfields $K^{gen}$ of the form $K\left(\sqrt{d_{1}},\sqrt{d_{2}},\sqrt{d_{3}}\right)$
with $d=d_{1}d_{2}d_{3}$. All unramified $H_{8}$-extensions of $K$
intersect $K^{gen}$ in one of these, though not all factorizations
can occur.

We remark that our method of normalizing by $g$ is reminiscent of
Gerth's idea \cite{gerthprank,gerth4rankidea} of replacing $Cl_{K}\left[p\right]$
by $Cl_{K}\left[p\right]/Cl_{K}\left[p\right]^{G}$ to obtain finite
moments in the abelian setting in the case when $p$ divides the degree
of $K$. There $Cl_{K}\left[p\right]^{G}$ corresponds to the whole
genus field. Gerth's conjecture was proven by Fouvry and Klüners for
$p=2$ \cite{fk1} and for $p=3$ by the second author also utilizing
similar methods \cite{klysdegreep}. Thus the case studied in this
paper should be considered the non-abelian analog of that situation.

\subsection{Acknowledgements}

The authors would like to thank Nigel Boston, Jacob Tsimerman and
Melanie Wood for helpful suggestions and encouragement, as well as
the organizers and participants of the 2016 Arithmetic Statistics
and Cohen-Lenstra conference at the University of Warwick where this
work began. The first author was supported by the National Science
Foundation grant DMS-1502553.

\section{Counting $H_{8}$ extensions}

First we restate the results of Lemmermeyer from \cite{lemmermeyerh8}
which are key to our expression for the number of unramified everywhere
$\left(H_{8}\rtimes C_{2},H_{8}\right)$ extensions of a quadratic
field $K$. 
\begin{thm}[Lemmermeyer]
Let $K$ be a quadratic field with discriminant $d$. There exists
an unramified $\left(H_{8}\rtimes C_{2},H_{8}\right)$-extension of
$K$ if and only if there exists a nontrivial factorization $d=d_{1}d_{2}d_{3}$
into relatively prime discriminants such that 
\[
\left(\frac{d_{1}d_{2}}{p_{3}}\right)=\left(\frac{d_{2}d_{3}}{p_{1}}\right)=\left(\frac{d_{3}d_{1}}{p_{2}}\right)=1
\]
for all $p_{i}\mid d_{i}$. 
\end{thm}
Let $K^{gen}=K\left(\sqrt{p_{1}},\ldots\sqrt{p_{r}}\right)$ where
$p_{i}$ are the prime fundamental discriminants dividing $d$. Any
unramified $H_{8}$-extension $L$ of $K$ will satisfy $L\cap K^{gen}=K\left(\sqrt{d_{1}},\sqrt{d_{2}},\sqrt{d_{3}}\right)$
where the $d_{i}$ are coprime discriminants and $d=d_{1}d_{2}d_{3}$.
Let $L^{g}=L\cap K^{gen}$. Then $L=L^{g}\left(\sqrt{\mu}\right)$
for some $\mu\in L^{g}$. 
\begin{prop}[Lemmermeyer]
\label{prop:Suppose--is}Suppose $L$ is an unramified $\left(H_{8}\rtimes C_{2},H_{8}\right)$-extension
of $K$ such that $L=L^{g}\left(\sqrt{\mu}\right)$. Then all the
other such extensions $M$ of $K$ with $M\cap K^{gen}=L^{g}$ are
exactly given by $L^{g}\left(\sqrt{\delta\mu}\right)$ for each $\delta\in\mathbb{Z}$
a discriminant such that $\delta\mid d$. 
\end{prop}
We use these results to prove the next proposition. 
\begin{prop}
\label{prop:The-number-of}The number of unramified everywhere $\left(H_{8}\rtimes C_{2},H_{8}\right)$-extensions
of $K$ is
\begin{align*}
f(d)= & \beta(d)2^{-3-\alpha(d)}\sum_{d=d_{1}d_{2}d_{3}}\prod_{p\mid d_{3}}\left(1+\Leg{d_{1}d_{2}}{p}\right)\prod_{p\mid d_{2}}\left(1+\Leg{d_{1}d_{3}}{p}\right)\prod_{p\mid d_{1}}\left(1+\Leg{d_{2}d_{3}}{p}\right)
\end{align*}
Where the sum is over nontrivial factorizations into coprime fundamental
discriminants $d=d_{1}d_{2}d_{3}$ up to permutation, and 
\begin{align*}
\alpha(d) & =\begin{cases}
1 & d>0\text{ and }\exists p\mid d,p\equiv3\mod4\\
0 & \text{else}
\end{cases}\\
\beta(d) & =\begin{cases}
1 & d<0\text{ or }\exists p\mid d,p\equiv3\mod4\\
\text{between }0\text{ and }1 & \text{else}
\end{cases}.
\end{align*}
\end{prop}
\begin{proof}
We have that 
\begin{align*}
2^{-\omega(d)}\prod_{p\mid d_{3}}\left(1+\Leg{d_{1}d_{2}}{p}\right)\prod_{p\mid d_{2}}\left(1+\Leg{d_{1}d_{3}}{p}\right)\prod_{p\mid d_{1}}\left(1+\Leg{d_{2}d_{3}}{p}\right)
\end{align*}
is $1$ if $d=d_{1}d_{2}d_{3}$ is an admissible factorization and
corresponds to an extension $L/K$ unramified at all finite places,
and $0$ otherwise. We have for each factorization $2^{\omega(d)-3}$
such extensions unramified at all finite places. For $d<0$, this
is the same as unramified everywhere and the result follows immeditaley.

Suppose $d>0$. Then we have $d_{1},d_{2},d_{3}>0$, and a map 
\begin{align*}
\varphi:Cl(K)[2]/\langle d_{1},d_{2},d_{3}\rangle\rightarrow\{+1,-1\}
\end{align*}
sending $d'\mapsto d'/|d'|$. If there exists a prime $p\mid d$ with
$p\equiv3\mod4$, then $-p\mid d$ is a discriminant, making $\varphi$
surjective. If $L^{g}(\sqrt{\mu})/K$ is an extension unramified at
all finite places, but ramified at the infinite place then $L^{g}(\sqrt{-p\mu})/K$
is unramified everywhere. We then have that the number of extensions
unramified everywhere is the size of the kernel, $2^{\omega(d)-4}$.
If there are no primes dividing $d$ congruent to $3\mod4$ then $\varphi$
is the trivial map. It follows similarly to the last case that either
every extension $L/K$ corresponding to the factorization $d=d_{1}d_{2}d_{3}$
is unramified at infinity or none of them are, so there are $\beta(d)2^{\omega(d)-3}$
such extensions for some $0\le\beta(d)\le1$ denoting the fraction
of such factorizations. 
\end{proof}
We can replace the condtions of Proposition \ref{prop:The-number-of}
by $\beta\left(d\right)=1$ for all $d$ and $\alpha\left(d\right)=1$
if $d>0$ and $0$ if $d<0$. This is because discriminants not divisible
by some $p\equiv3\mod4$ do not contribute to the main term, which
can be seen by an argument similar to Corollary 1 from \cite{fk1}
which we sketch here.
\begin{lem}
Let
\begin{align*}
\widetilde{f}(d) & =2^{-3-\widetilde{\alpha}(d)}\sum_{d=d_{1}d_{2}d_{3}}\prod_{p\mid d_{3}}\left(1+\Leg{d_{1}d_{2}}{p}\right)\prod_{p\mid d_{2}}\left(1+\Leg{d_{1}d_{3}}{p}\right)\prod_{p\mid d_{1}}\left(1+\Leg{d_{2}d_{3}}{p}\right),\\
\widetilde{\alpha}(d) & =\begin{cases}
1 & d>0\\
0 & d<0.
\end{cases}
\end{align*}
If Theorem \ref{thm:mainthm3} holds for $\widetilde{f}(d)$, then
it also holds for $f(d)$. 
\end{lem}
\begin{proof}
We have by Landau's Theorem and by H$\ddot{\text{o}}$lder's inequality
with $1/b+1/c=1$ that
\begin{align*}
\sum_{d\in\mathcal{D}_{X,1,4}^{\pm},p\mid d\Rightarrow p\equiv1\text{mod}4}\left(a^{\omega(d)}f(d)\right)^{k} & =\left(\sum_{d\in\mathcal{D}_{X,1,4}^{\pm},p\mid d\Rightarrow p\equiv1\text{mod}4}1\right)^{1/b}\left(\sum_{d\in\mathcal{D}_{X,1,4}^{\pm}}\left(a^{\omega(d)}f(d)\right)^{k}\right)^{1/c}\\
 & \ll\left(\frac{X}{\sqrt{\log X}}\right)^{1/b}\left(X\left(\log X\right)^{(3a)^{k}-1}\right)^{1/c}.
\end{align*}
The result follows by noting that we can find values $b,c>1$ such
that $-\frac{1}{2b}+\frac{(3a)^{k}-1}{c}\le(3a)^{k}-a^{k}-1+\epsilon$,
forcing the only terms where the two sums differ into the error term
(for example, taking $b=a^{-k}((3a)^{k}-1/2)>1$). 
\end{proof}
Thus in the remainder of the paper we will let write $f\left(d\right)$
to denote $\widetilde{f}(d)$. We will proceed with the proof of Theorem
\ref{thm:mainthm3}. Our first step will be to express 
\[
S_{k}^{\pm}=\sum_{d\in\mathcal{D}_{X}^{\pm}}\left(a^{\omega(d)}f(d)\right)^{k}
\]
as a character sum. We note that below the factorizations will not
be assumed to be nontrivial and the sums will be over all permutations
of the $d_{i}$.

As previously mentioned we will first prove the results in the case
of negative discriminants congruent to $1\mod4$. The computations
necessary to treat the remaining cases are handled in the appendix.

For a quadratic field $K$ with fundamental discriminant $d$ define
$f\left(d\right)$ to be the number of unramified $H_{8}$-extensions
of $K$ which are Galois over $\Q$. Define functions

\begin{eqnarray*}
f_{1}\left(d\right) & = & 2^{-3}\frac{1}{2}\sum_{d=-d_{1}d_{2}d_{3}}\prod_{p\mid d_{3}}\left(1+\left(\frac{-d_{1}d_{2}}{p}\right)\right)\prod_{p\mid d_{1}}\left(1+\left(\frac{d_{2}d_{3}}{p}\right)\right)\prod_{p\mid d_{2}}\left(1+\left(\frac{-d_{3}d_{1}}{p}\right)\right)
\end{eqnarray*}
and

\begin{eqnarray*}
f_{2}\left(d\right) & = & 2^{-3}\sum_{d=-d_{1}d_{2}}\prod_{p\mid d_{1}}\left(1+\left(\frac{d_{2}}{p}\right)\right)\prod_{p\mid d_{2}}\left(1+\left(\frac{-d_{1}}{p}\right)\right)
\end{eqnarray*}
where the factoriation $d=-d_{1}d_{2}d_{3}$ is into integers such
that $d_{1}\equiv-1\mod4$ and $d_{2},d_{3}\equiv1\mod4$ and similarly
for $d=d_{1}d_{2}$ is into integers such that $d_{1}\equiv-1\mod4$
and $d_{2}\equiv1\mod4$.

Then we have 
\[
f\left(d\right)=f_{1}\left(d\right)-f_{2}\left(d\right)+2^{\omega(d)-3}.
\]
The congruence conditions above are imposed since the condition from
\cite{lemmermeyerh8} requires that the factorization be into fundamental
discriminants, but factorizations where more than one term is negative
don't contiribute to the count (that is the expression in \cite{lemmermeyerh8}
evaluates to 0). So we can assume $-d_{1}$ is the negative factor
in the factorization of $d$. The expression contains $f_{2}$ and
$2^{\omega(d)-4}$ to account for trivial factorizations (where at
least one of the $d_{i}$ equals 1). Note this implies $f\left(d\right)=0$
when $\omega\left(d\right)\le2$.

Expanding these expressions gives 
\begin{align}
f_{1}\left(d\right) & =\frac{1}{2^{4}}\sum_{d=-d_{1}d_{2}d_{3}}\left(\sum_{a\mid d_{3}}\left(\frac{-d_{1}d_{2}}{a}\right)\right)\left(\sum_{b\mid d_{1}}\left(\frac{-d_{2}d_{3}}{b}\right)\right)\left(\sum_{c\mid d_{2}}\left(\frac{-d_{3}d_{1}}{c}\right)\right)\nonumber \\
 & =\frac{1}{2^{4}}\sum_{d=-d_{1}d_{2}d_{3}}\left(\sum_{D_{4}D_{1}=d_{3}}\left(\frac{-D_{0}D_{3}D_{2}D_{5}}{D_{4}}\right)\right)\left(\sum_{D_{0}D_{3}=d_{1}}\left(\frac{D_{2}D_{5}D_{4}D_{1}}{D_{0}}\right)\right)\nonumber \\
 & \times\left(\sum_{D_{2}D_{5}=d_{2}}\left(\frac{-D_{4}D_{1}D_{0}D_{3}}{D_{2}}\right)\right)\nonumber \\
 & =\frac{1}{2^{4}}\sum_{d=-D_{0}D_{1}D_{2}D_{3}D_{4}D_{5}}\left(\frac{-1}{D_{4}}\right)\left(\frac{-1}{D_{2}}\right)\nonumber \\
 & \times\left(\frac{D_{0}D_{3}D_{2}D_{5}}{D_{4}}\right)\left(\frac{D_{2}D_{5}D_{4}D_{1}}{D_{0}}\right)\left(\frac{D_{4}D_{1}D_{0}D_{3}}{D_{2}}\right)\label{eq:negative1mod4}
\end{align}
and

\begin{align}
f_{2}\left(d\right) & =\frac{1}{2^{3}}\sum_{d=-d_{1}d_{2}}\left(\sum_{b\mid d_{1}}\left(\frac{d_{2}}{b}\right)\right)\left(\sum_{c\mid d_{2}}\left(\frac{-d_{1}}{c}\right)\right)\nonumber \\
 & =\frac{1}{2^{3}}\sum_{d=-d_{1}d_{2}}\left(\sum_{E_{0}E_{1}=d_{1}}\left(\frac{E_{2}E_{3}}{E_{0}}\right)\right)\left(\sum_{E_{2}E_{3}=d_{2}}\left(\frac{-E_{0}E_{1}}{E_{2}}\right)\right)\nonumber \\
 & =\frac{1}{2^{3}}\sum_{d=-E_{0}E_{1}E_{2}E_{3}}\left(\frac{-1}{E_{2}}\right)\left(\frac{E_{2}E_{3}}{E_{0}}\right)\left(\frac{E_{0}E_{1}}{E_{2}}\right).\label{eq:errornegative1mod4}
\end{align}
where the sums are over factorizations which satisfy the congruences
\begin{eqnarray*}
D_{0}D_{3} & \equiv & -1\mod4\\
D_{2}D_{5} & \equiv & 1\mod4\\
D_{4}D_{1} & \equiv & 1\mod4
\end{eqnarray*}
and 
\begin{eqnarray*}
E_{0}E_{1} & \equiv & -1\mod4\\
E_{2}E_{3} & \equiv & 1\mod4.
\end{eqnarray*}
For two indices $u,v$ define the function $\Phi\left(u,v\right)\in\mathbb{F}_{2}$
to be 1 if and only if the symbol $\left(\frac{D_{u}}{D_{v}}\right)$
appears in $\left(\ref{eq:negative1mod4}\right)$ and similarly define
$\lambda\left(u\right)\in\mathbb{F}_{2}$ to be 1 if and only if $\left(\frac{-1}{D_{u}}\right)$
appears. Also define the function $\Psi\left(u,v\right)\in\mathbb{F}_{2}$
to be 1 if and only if the symbol $\left(\frac{E_{u}}{E_{v}}\right)$
appears in $\left(\ref{eq:errornegative1mod4}\right)$ and similarly
define $\gamma\left(u\right)\in\mathbb{F}_{2}$ to be 1 if and only
if $\left(\frac{-1}{E_{u}}\right)$ appears.

Taking the $k$th power of $f\left(d\right)$ gives 
\begin{eqnarray*}
f\left(d\right)^{k} & = & \sum_{j_{1}+j_{2}+j_{3}=k}^{k}\binom{k}{j_{1},j_{2},j_{3}}f_{1}^{j_{1}}\left(d\right)\left(-1\right)^{j_{2}}f_{2}^{j_{2}}\left(d\right)2^{j_{3}(\omega(d)-4)}\\
 & = & \sum_{j_{1}+j_{2}+j_{3}=k}^{k}\binom{k}{j_{1},j_{2},j_{3}}\left(-1\right)^{j_{2}}2^{j_{3}(\omega(d)-4)}f_{j_{1},j_{2}}\left(d\right).
\end{eqnarray*}

We will define some notation to rewrite $f_{j_{1},j_{2}}\left(d\right)$
in a form suitable to application of analytic techniques. In the expression
for $f_{j_{1},j_{2}}\left(d\right)$ we have $j_{1}$ different factorizations
of $d$ into 6 variables and $j_{2}$ factorizations of $d$ into
$4$ variables, with $j_{1}+j_{2}=l$. Write these factorizations
of $d$ as 
\begin{eqnarray*}
d & = & \prod_{u_{1}}D_{u_{1}}^{\left(1\right)}=\cdots=\prod_{u_{j_{1}}}D_{u_{j_{1}}}^{\left(j_{1}\right)}\\
 & = & \prod_{u_{j_{1}+1}}E_{j_{1}+1}^{\left(j_{1}+1\right)}=\cdots=\prod_{u_{l}}E_{u_{l}}^{\left(l\right)}
\end{eqnarray*}
where each index $u_{i}$ runs from $0$ to $5$ for $1\le i\le j_{1}$
and from 0 to 3 for $j_{1}+1\le i\le l$. From this we obtain a further
factorization of each $D_{u_{h}}^{\left(h\right)}$ by 
\[
D_{u_{h}}^{\left(h\right)}=\prod_{1\le i\le l,i\neq h}\prod_{u_{h}}\gcd\left(D_{u_{1}}^{\left(1\right)},\ldots,D_{u_{j_{1}}}^{\left(j_{1}\right)},E_{u_{j_{1}+1}}^{\left(j_{1}+1\right)},\ldots,E_{u_{l}}^{\left(l\right)}\right).
\]
Define 
\[
D_{u_{1},\ldots,u_{j_{2}}}=\gcd\left(D_{u_{1}}^{\left(1\right)},\ldots,D_{u_{j_{1}}}^{\left(j_{1}\right)},E_{j_{1}+1}^{\left(j_{1}+1\right)},\ldots,E_{u_{l}}^{\left(l\right)}\right).
\]

To simplify notation we let $u=\left(u_{1},\ldots,u_{l}\right)$,
$v=\left(v_{1},\ldots,v_{l}\right)$, and define the functions 
\[
\Phi_{j_{1}}\left(u,v\right)=\sum_{i=1}^{j_{1}}\Phi\left(u_{i},v_{i}\right)
\]
and 
\[
\Psi_{j_{2}}\left(u,v\right)=\sum_{i=j_{1}}^{l}\Psi\left(u_{i},v_{i}\right).
\]
Thus with the above notations we can write 
\[
f_{j_{1},j_{2}}\left(d\right)=\frac{1}{2^{4j_{1}+3j_{2}}}\sum_{\left(D_{u}\right)}\prod_{u}\left(\frac{-1}{D_{u}}\right)^{\lambda_{j_{1}}\left(u\right)+\gamma_{j_{2}}\left(u\right)}\prod_{u,v}\left(\frac{D_{u}}{D_{v}}\right)^{\Phi_{j_{1}}\left(u,v\right)+\Psi_{j_{2}}\left(u,v\right)}
\]
where now the sum is over $6^{j_{1}}4^{j_{2}}$ tuples of integers
$\left(D_{u}\right)$ which satisfy $\prod_{u}D_{u}=d$ and the following
congruence conditions:

for all $1\le i\le j_{1}$ and $\left(u_{i},v_{i}\right)\in\left\{ \left(0,3\right),\left(2,5\right),\left(4,1\right)\right\} $

\begin{equation}
\prod_{u}D_{u}\prod_{v}D_{v}\equiv\begin{cases}
-1\mod4 & \mbox{if }\left(u_{i},v_{i}\right)=\left(0,3\right)\\
1\mod4 & \mbox{if }\left(u_{i},v_{i}\right)=\left(2,5\right),\left(4,1\right)
\end{cases}\label{eq:negative1mod4congruence1}
\end{equation}
and for all $j_{1}+1\le i\le l$ and $\left(u_{i},v_{i}\right)\in\left\{ \left(0,1\right),\left(2,3\right)\right\} $
\begin{equation}
\prod_{u}D_{u}\prod_{v}D_{v}\equiv\begin{cases}
-1\mod4 & \mbox{if }\left(u_{i},v_{i}\right)=\left(0,1\right)\\
1\mod4 & \mbox{if }\left(u_{i},v_{i}\right)=\left(2,3\right)
\end{cases}\label{eq:negative1mod4congruence2}
\end{equation}
where the above products are over all $u$ with $u_{i}$ in the $i$th
position and all $v$ with $v_{i}$ in the $i$th position.

Thus multiplying by $a^{k\omega\left(d\right)}$ and summing over
discriminants $d<0$ with $d\equiv1\mod4$ we get 
\begin{eqnarray}
\sum_{d<X}2^{j_{3}\omega(d)}a^{k\omega\left(d\right)}f_{j_{1},j_{2}}\left(d\right) & = & \frac{1}{2^{4j_{1}+3j_{2}}}\sum_{\left(D_{u}\right)}\mu^{2}\left(\prod_{u}D_{u}\right)2^{j_{3}\omega(d)}a^{k\omega\left(d\right)}\label{eq:mainexpression1}\\
 &  & \times\prod_{u}\left(\frac{-1}{D_{u}}\right)^{\lambda_{j_{1}}\left(u\right)+\gamma_{j_{2}}\left(u\right)}\prod_{u,v}\left(\frac{D_{u}}{D_{v}}\right)^{\Phi_{j_{1}}\left(u,v\right)+\Psi_{j_{2}}\left(u,v\right)}\nonumber 
\end{eqnarray}
where the sum is over $6^{j_{1}}4^{j_{2}}$ tuples of integers $\left(D_{u}\right)$
which satisfy $\prod_{u}D_{u}<X$ and the conditions $\left(\ref{eq:negative1mod4congruence1}\right)$
and $\left(\ref{eq:negative1mod4congruence2}\right)$.

\section{\label{sec:Bounding-the-Error}Bounding the Error Term}

So far we have written 
\[
S_{k}\left(X\right)=\sum_{j_{1}+j_{2}+j_{3}=k}^{k}\binom{k}{j_{1},j_{2},j_{3}}\left(-1\right)^{j_{2}}a^{k\omega(d)}2^{j_{3}(\omega(d)-3-\alpha(d))}f_{j_{1},j_{2}}\left(d\right).
\]
Where $\alpha(d)=1$ if $d<0$ or there exists a prime $p\equiv3\mod4$
with $p\mid d$, and $\alpha(d)=0$ otherwise. We will analyse each
term seperately. Let 
\[
S_{j_{1},j_{2},j_{3}}\left(X\right)=\sum_{d<X}a^{k\omega(d)}2^{j_{3}(\omega(d)-3-\alpha(d))}f_{j_{1},j_{2}}\left(d\right).
\]
The right hand side was computed in the previous section and is given
by $\left(\ref{eq:mainexpression1}\right)$.

We want to separate $\left(\ref{eq:mainexpression1}\right)$ into
a main term and error term. The methods of Fouvry-Klüners from \cite{fk1}
apply with minor modifications to accomplish this. We state the key
points of their argument and refer to \cite{fk1} for the proofs.

Fix $k\in\mathbb{Z}_{\ge1}$ and let $\Delta=1+\log^{-(a6)^{k}}X$.
Define $\mathbf{A}$ to be a tuple $\left(A_{i}\right)_{i=0}^{6^{j_{1}}4^{j_{2}}}$
of variables with each $A_{i}$ corresponding to $D_{i}$, and each
$A_{i}=\Delta^{j}$ for some $j\ge0$. We can partition $S_{j_{1},j_{2},j_{3}}\left(X\right)$
according to the various $\mathbf{A}$, by letting $S_{j_{1},j_{2},j_{3}}\left(X,\mathbf{A}\right)$
be the sum $\left(\ref{eq:mainexpression1}\right)$ but now restricted
to tuples $\left(D_{i}\right)$ for which $A_{i}\le D_{i}\le\Delta A_{i}$
and $\prod_{i}D_{i}<X$. Hence 
\[
S_{j_{1},j_{2},j_{3}}\left(X\right)=\sum_{\mathbf{A}}S_{j_{1},j_{2},j_{3}}\left(X,\mathbf{A}\right).
\]

Note that if $\Delta=1+\log^{-(a6)^{k}}X$ then there are $O\left(\left(\log X\right)^{6^{j_{1}}4^{j_{2}}\left(1+(a6)^{k}\right)}\right)$
possible $\mathbf{A}$ with $S_{j_{1},j_{2}}\left(X,\mathbf{A}\right)$
not empty. This is since there are $O\left(\left(\log X\right)^{\left(1+(a6)^{k}\right)}\right)$
choices for each $1<A_{i}\le X$.

Let $\Omega=e(a6)^{k}\left(\log\log X+B_{0}\right).$ Noting $a>1/6$,
we have 
\begin{lem}[Fouvry-Klüners \cite{fk1}]
Let $S$ be the sum of the terms in $\left(\ref{eq:mainexpression1}\right)$
which satisfy: at least one $d_{i}$ has $\omega\left(d_{i}\right)>\Omega$.
Then 
\[
S=O\left(X\left(\log X\right)^{-1}\right).
\]
\begin{lem}[Fouvry-Klüners \cite{fk1}]
Let $\mathcal{F}_{1}$ be the set of $\mathbf{A}$ which satisfies
$\prod_{i}\Delta A_{i}>X$. Then 
\[
\sum_{\mathbf{A}\in\mathcal{F}_{1}}S_{j_{1},j_{2},j_{3}}\left(X,\mathbf{A}\right)=O\left(X\left(\log X\right)^{-1}\right).
\]
\end{lem}
\end{lem}
For the next two lemmas we will need to define 
\[
X^{\dagger}=\log^{3\left(1+6^{j_{1}}4^{j_{2}}\left(1+(a6)^{k}\right)\right)}X
\]
\[
X^{\ddagger}=\exp\left(\log^{\eta(k)}X\right)
\]
for some small $\eta(k)$ depending only on $k$.

Now let $Y_{1}=\left\{ 0,1,2,3,4,5\right\} $ and let $Y_{2}=\left\{ 0,1,2,3\right\} $.
The indices of the variables lie in $Y=Y_{1}^{j_{1}}\times Y_{2}^{j_{2}}$.We
define two indices $u,v\in Y$ to be linked if $\Phi\left(u,v\right)+\Psi\left(u,v\right)+\Phi\left(v,u\right)+\Psi\left(v,u\right)=1$.
This means that exactly one of the symbols $\left(\frac{D_{u}}{D_{v}}\right)$
and $\left(\frac{D_{v}}{D_{u}}\right)$ appears in $\left(\ref{eq:mainexpression1}\right)$.

Define 
\[
U\left(j_{3}\right)=\begin{cases}
3^{k-j_{3}}+1 & if\ j_{3}>0\\
3^{k} & if\ j_{3}=0
\end{cases}
\]

\begin{lem}[Fouvry-Klüners \cite{fk1}]
\label{lem:-secondfamily} Let $\mathcal{F}_{2}$ be the set of $\mathbf{A}$
satisfiying: at most $U\left(j_{3}\right)-1$ variables are $A_{i}>X^{\ddagger}$.
Then 
\[
\sum_{\mathbf{A}\in\mathcal{F}_{2}}S_{j_{1},j_{2},j_{3}}\left(X,\mathbf{A}\right)\ll X\left(\log X\right)^{2^{j_{3}}a^{k}\left(3^{k-j_{3}}-1+\eta(k)6^{j_{1}}4^{j_{2}}\right)-1}
\]
\begin{lem}[Fouvry-Klüners \cite{fk1}]
Let $\mathcal{F}_{3}$ be the set of $\mathbf{A}$ satisfying: there
exist two linked indices $i$ and $j$ with $A_{i}\ge X^{\ddagger}$
and $A_{j}\ge2$. Then 
\[
\sum_{\mathbf{A}\in\mathcal{F}_{3}}S_{j_{1},j_{2},j_{3}}\left(X,\mathbf{A}\right)=O\left(X\left(\log X\right)^{-1}\right).
\]
\end{lem}
\end{lem}
Consider the set of $\mathbf{A}$ satisfying 
\begin{equation}
\mathbf{A}\mbox{ is not in \ensuremath{\mathcal{F}_{i}} for any \ensuremath{i=0,1,2,3}.}\label{eq:familycondition}
\end{equation}
Combining the above lemmas we reduce our expression to 
\[
S_{j_{1},j_{2},j_{3}}\left(X\right)=\sideset{}{^{'}}\sum_{\mathbf{A}}S_{j_{1},j_{2},j_{3}}\left(X,\mathbf{A}\right)+O\left(X\left(\log X\right)^{2^{j_{3}}a^{k}3^{k-j_{3}}-2^{j_{3}}a^{k}-1+\epsilon}\right)
\]
where the sum is over $\mathbf{A}$ satisfying $\left(\ref{eq:familycondition}\right)$,
after taking $\eta(k)=\epsilon6^{-j_{1}}4^{-j_{2}}2^{-j_{3}}a^{-k}$.

Note this condition implies that there are at least $U\left(j_{3}\right)$
variables $A_{i}>X^{\ddagger}$ and they are all unlinked. In the
next section we will show that a maximal unlinked set in $Y$ is exactly
of size $3^{j_{1}}2^{j_{2}}$ and this is strictly less than $U\left(j_{3}\right)$
unless $j_{1}=k$.

\section{Maximal unlinked sets}

Consider the set of indices $Y=Y_{1}^{j_{1}}\times Y_{2}^{j_{2}}$.
As in the previous section, we call two indices $u,v$ linked if $\Phi_{j_{1}}(u,v)+\Psi_{j_{2}}(u,v)+\Phi_{j_{1}}(v,u)+\Psi_{j_{2}}(v,u)=1$
and unlinked otherwise. When $j_{1}=0$ the maximal unlinked subsets
of $Y$ are determined in \cite{fk1} and are of size $2^{j_{2}}$.
We will now determine the largest maximal unlinked sets when $j_{2}=0.$ 
\begin{prop}
\label{thm:maxunlinkedsets}Let $A=\left\{ 1,3,5\right\} $ and $B=\left\{ 0,2,4\right\} $.
Let $S=\left\{ A,B\right\} ^{k}$. The largest maximal unlinked sets
are all of size $3^{k}$ and correspond bijectively to elements of
$S$. The set corresponding to $s\in S$ is 
\[
U_{s}=\left\{ u\in Y^{k}\mid u_{i}\in s_{i}\right\} .
\]
\end{prop}
\begin{proof}
Define a graph $G_{k}$ with vertices $\{0,1,2,3,4,5\}^{k}$ and adjacency
matrix given by $[G_{k}]\equiv[B_{k}(u,v)]\mod2$, where we define
$B_{k}(u,v)=\Phi_{k}(u,v)+\Phi_{k}(v,u)$. Unlinked sets are exactly
the independent sets of $G_{k}$. Notice for $k=1$ that $G_{1}$
is a cyclic graph with $6$ vertices, and has largest maximal independent
sets given by $A$ and $B$. We use this as a base case for induction.

Suppose the theorem holds true for $k-1$, and let $U\subset G_{k}$
be independent, and partition it into $U=\coprod_{i=0}^{5}C_{i}$
where $C_{i}=\{(u,i)\in U:u\in G_{k-1}\}$. Call $c_{i}=|C_{i}|$,
so that we have $|U|=\sum c_{i}$. We know that 
\[
[G_{k}]\equiv[B_{k}(u,v)]=[B_{k-1}((u_{2},...,u_{k}),(v_{2},...,v_{k}))]+[B_{1}(u_{1},v_{1})]
\]
The subgraph induced by $U$ inside of $G_{k}$ corresponds to a submatrix
$[U]=0$ in $[G_{k}]$ along the vertices of $U$. In particular,
\[
[B_{k-1}((u_{2},...,u_{k}),(v_{2},...,v_{k}))]=[B_{1}(i,j)]
\]
for all $u\in C_{i}$, $v\in C_{j}$. Ordering indices lexicogrphically,
with the $k^{th}$ entire weighted highest, we then have 
\[
[B_{k-1}|_{U}]=\begin{pmatrix}0_{c_{0}\times c_{0}} & 1_{c_{0}\times c_{1}} & 0_{c_{0}\times c_{2}} & 0_{c_{0}\times c_{3}} & 0_{c_{0}\times c_{4}} & 1_{c_{0}\times c_{5}}\\
1_{c_{1}\times c_{0}} & 0_{c_{1}\times c_{1}} & 1_{c_{1}\times c_{2}} & 0_{c_{1}\times c_{3}} & 0_{c_{1}\times c_{4}} & 0_{c_{1}\times c_{5}}\\
0_{c_{2}\times c_{0}} & 1_{c_{2}\times c_{1}} & 0_{c_{2}\times c_{2}} & 1_{c_{2}\times c_{3}} & 0_{c_{2}\times c_{4}} & 0_{c_{2}\times c_{5}}\\
0_{c_{3}\times c_{0}} & 0_{c_{3}\times c_{1}} & 1_{c_{3}\times c_{2}} & 0_{c_{3}\times c_{3}} & 1_{c_{3}\times c_{4}} & 0_{c_{3}\times c_{5}}\\
0_{c_{4}\times c_{0}} & 1_{c_{4}\times c_{1}} & 0_{c_{4}\times c_{2}} & 1_{c_{4}\times c_{3}} & 0_{c_{0}\times c_{4}} & 1_{c_{4}\times c_{5}}\\
1_{c_{5}\times c_{0}} & 0_{c_{5}\times c_{1}} & 0_{c_{5}\times c_{2}} & 0_{c_{5}\times c_{3}} & 1_{c_{5}\times c_{4}} & 0_{c_{5}\times c_{5}}
\end{pmatrix}
\]
Where $0_{n\times m}$ and $1_{n\times m}$ are block matrices of
dimension $n\times m$ with all $0$ and $1$ entries respectively. 
\begin{lem}
\begin{equation}
c_{i-1}+c_{i+1}\le\begin{cases}
2\cdot3^{k-1} & c_{i}=0,c_{i+2}=c_{i+4}=0\\
3^{k-1} & c_{i}=0;else\\
2(3^{k-1}-1) & c_{i}\ne0,c_{i+2}=c_{i+4}=0\\
3^{k-1}-1 & c_{i}\ne0;else
\end{cases}\label{eq:twoabounds}
\end{equation}
\end{lem}
\begin{proof}
Consider a complete bipartite induces subgraph $K_{V,W}\subset G_{k-1}$.
By the inductive hypothesis $|V|\le3^{k-1}$ with equality if and
only if $V=U_{t}$ for some $t\in S$, and similarly for $W$. Suppose
$V\ne\emptyset$, then for any $u\in V$ of $t\in S$ we have that
there exists a $v\in U_{s}$ for any $s\in S$ defined 
\[
u_{j}=\begin{cases}
v_{j+2m} & t_{j}=s_{j}\\
v_{j+3} & t_{j}\ne s_{j}
\end{cases}
\]
Then $B_{k-1}(u,v)=0$, so $v\not\in W$. Thus $W\ne U_{s}$ and so
$|W|\le3^{k-1}-1$. By symmetry the same is true for $V$ if $W\ne\emptyset$.

Define $p:G_{k}\rightarrow G_{k-1}$ to be the projection forgetting
the $k^{th}$ coordinate. Notice that $p|_{C_{i}}$ is injective for
all $i$ values.

Suppose $c_{i}=0,c_{i+2}=c_{i+4}=0$. Then we use the trivial bound:
the submatrix on vertices in $C_{j}$ is a block zero matrix, so that
$p(C_{j})$ is an independent set of $G_{k-1}$. Thus, $c_{j}\le3^{k-1}$,
so $c_{i-1}+c_{i+1}\le2\cdot3^{k-1}$ by the inductive hypothesis.

Suppose $c_{i}=0$ and without loss of generality $c_{i+2}\ne0$.
Then for $(u,i-1)\in C_{i-1}$ and $(v,i+1)\in C_{i+1}$, choose some
$(w,i+2)\in C_{i+2}$. We have $B_{k-1}(u,w)=B(i-1,i+2)=0$ and $B_{k-1}(v,w)=B(i+1,i+2)=1$,
implying $u\ne v$. Then we have $p(C_{i-1})\cap p(C_{i+1})=\emptyset$
and $p(C_{i-1})\cup p(C_{i+1})$ is an independent set of $G_{k-1}$,
by $B_{k-1}(u,v)=B(i-1,i+1)=0$. Thus $c_{i-1}+c_{i+1}\le3^{k-1}$
by the inductive hypothesis.

Suppose $c_{i}\ne0,c_{i+2}=c_{i+4}=0$. Then $p(C_{i+1})\cup p(C_{i})$
is a complete bipartite induced subgraph of $G_{k-1}$, $K_{V,W}$
for $|V|=c_{i+1}$ and $|W|=c_{i}$. Similarly for $c_{i-1},c_{i}$.
Then $c_{i-1},c_{i+1}\le3^{k-1}-1$ and the result follows.

Suppose $c_{i}\ne0$ and without loss of generality $c_{i+2}\ne0$.
We can similarly prove $p(C_{i-1})\cap p(C_{i+1})=\emptyset$ and
$p(C_{i-1})\cup p(C_{i+1})\cup p(C_{i})$ is an induce bipartite subgraph
of $G_{k-1}$, $K_{V,W}$ with $V=p(C_{i-1})\cup p(C_{i+1})$ and
$W=p(C_{i})$. So a combination of the previous two results gives
us this case. 
\end{proof}
Let $I=\{i:c_{i}=0\}$ we will separate cases based on the size of
$I$.

If $|I|\ge4$ then we have 
\begin{align*}
2|U| & =\sum_{i}c_{i-1}+c_{i+1}=\sum_{i\not\in I}2c_{i}\\
 & \le2(6-|I|)3^{k-1}\\
 & \le4\cdot3^{k-1}\\
 & <2\cdot3^{k}
\end{align*}
And is not maximum.

If $|I|=3$ we must separate into two cases. If $I=\{0,2,4\}$ or
$\{1,3,5\}$ then $U\le3^{k}$ with equality iff $p(C_{j})$ is of
maximum size for an indeendent set in $G_{k-1}$ for all $j\not\in I$.
But maximum implies it equals some $U_{s}$, and $I=\{0,2,4\}$ or
$\{1,3,5\}$ implies we can extend the type for $k-1$ to a type for
$k$ with $U=U_{s}$.

Otherwise, by symmetry we can assume $I\cap\{0,2,4\}=\{0,2\}$. Let
$j$ be the third element of $I$. Then we have by the above lemma
$c_{1}+c_{3},c_{1}+c_{5}\le3^{k-1}$. We also have $c_{0}+c_{2}=0$
and $c_{0}+c_{4}=c_{2}+c_{4}=c_{4}\le3^{k-1}-1$ by at least one of
$c_{5},c_{3}$ nonzero. Lastly $c_{3}+c_{5}\le2\cdot3^{k-1}$. Thus
we have 
\begin{align*}
2|U| & =\sum_{i}c_{i-1}+c_{i+1}\\
 & \le4\cdot3^{k-1}+2(3^{k-1}-1)\\
 & \le2\cdot3^{k}-2\\
 & <2\cdot3^{k}
\end{align*}
And is then not maximum.

If $|I|=2$ we need two cases. First, if $I\subset\{0,2,4\}$ or $\{1,3,5\}$.
Then there exists a $j$ such that $I=\{j-1,j+1\}$. Without loss
of generality suppose $j=1$. Then we have $c_{0}+c_{2}=0$ and $c_{2}+c_{4}=c_{0}+c_{4}\le3^{k-1}-1$.
We also have $c_{1}+c_{3},c_{1}+c_{5}\le3^{k-1}$ and $c_{3}+c_{5}\le2\cdot(3^{k-1}-1)$.
Thus we have 
\begin{align*}
2|U| & =\sum_{i}c_{i-1}+c_{i+1}\\
 & \le2\cdot(3^{k-1}-1)+2\cdot3^{k-1}+2(3^{k-1}-1)\\
 & \le2\cdot3^{k}-4\\
 & <2\cdot3^{k}
\end{align*}
And is not maximum.

Otherwise, $I\ne\{j-1,j+1\}$. Then we have $c_{j-1}+c_{j+1}<3^{k-1}$
if $j\in I$ and $3^{k-1}-1$ is $j\not\in I$. Thus we have 
\begin{align*}
2|U| & =\sum_{i}c_{i-1}+c_{i+1}\\
 & \le2\cdot3^{k-1}+4(3^{k-1}-1)\\
 & \le2\cdot3^{k}-4\\
 & <2\cdot3^{k}
\end{align*}
And is not maximum.

If $|I|=1$ suppose without loss of generality that $0\in I$. Then
$c_{0}+c_{2},c_{0}+c_{4},c_{1}+c_{5}\le3^{k-1}$ and $c_{1}+c_{3},c_{2}+c_{4},c_{3}+c_{5}\le3^{k-1}-1$.
Thus we have 
\begin{align*}
2|U| & =\sum_{i}c_{i-1}+c_{i+1}\\
 & \le3^{k}+3(3^{k-1}-1)\\
 & \le2\cdot3^{k}-3\\
 & <2\cdot3^{k}
\end{align*}
And is not maximum.

If $|I|=0$ then $c_{i-1}+c_{i+1}\le3^{k-1}-1$. Thus we have 
\begin{align*}
2|U| & =\sum_{i}c_{i-1}+c_{i+1}\\
 & \le6(3^{k-1}-1)\\
 & \le2\cdot3^{k}-6\\
 & <2\cdot3^{k}
\end{align*}
And is not maximum. 
\end{proof}
To simplify notation we will refer to the largest maximal unlinked
set corresponding to $s\in S$ as being of type $s$ or simply as
a type.

We now combine these results to determine the largest maximal unlinked
sets for all $j_{1},j_{2}>0$. 
\begin{prop}
The largest maximal unlinked sets in $Y$ are of the form $V\times W$
where $V$ is a type in $Y_{1}^{j_{1}}$ and $W$ is a maximal unlinked
set in $Y_{2}^{j_{2}}$. Thus the largest maximal unlinked sets of
$Y$ are of size $3^{j_{1}}2^{j_{2}}$.\label{thm:trivialunlinked} 
\end{prop}
\begin{proof}
We fix $j_{1}>0$ and prove this by induction on $j_{2}$. Let $G_{j_{1},j_{2}}=Y_{1}^{j_{1}}\times Y_{2}^{j_{2}}$.
The base case $j_{2}=0$ is Theorem \ref{thm:maxunlinkedsets}.

Let $U\subset G_{j_{1},j_{2}}$ be a largest maximal unlinked set.
Now let $C_{i}=\left\{ \left(u,i\right)\in U\mid u\in G_{j_{2}-1}\right\} $
and $c_{i}=|C_{i}|$ for $i\in Y_{2}$, as above. Let $p:G_{j_{1},j_{2}}\longrightarrow G_{j_{1},j_{2}-1}$
be the projection dropping the last coordinate. Suppose $i$ and $j$
are unlinked, and $c_{i},c_{j},c_{k}\ne0$ for $i,j,k$ all distinct.
Let $(u,i)\in C_{i}$, $(v,j)\in C_{j}$, and $(w,k)\in C_{k}$. $k$
is linked to exactly one of $i,j$ by the pigeonhole principle, as
$i,j,k\in\{0,1,2,3\}$ and the linked pairs are $\{0,1\}$ and $\{2,3\}$,
so without loss of generality say $k$ and $i$ are linked. Then it
follows that $B_{j_{1},j_{2}-1}(u,w)=B_{0,1}(i,k)=1$ and $B_{j_{1},j_{2}-1}(v,w)=B_{0,1}(j,k)=0$.
So we must have $u\ne v$, which implies $p$ is injective on $C_{i}\cup C_{j}$.

We consider several cases:

\textbf{Case 1: }Suppose only one $c_{i}$ is nonzero. Note $p\left(C_{i}\right)$
is unlinked for any $i$. Hence $c_{i}\le3^{j_{1}}2^{j_{2}-1}<3^{j_{1}}2^{j_{2}}$
by the inductive hypothesis.

\textbf{Case 2: }Suppose exactly two $c_{i}$ and $c_{j}$ are non-zero.
Suppose $i$ and $j$ are linked. If $c_{i}=3^{j_{1}}2^{j_{2}-1}$
then by the induction hypothesis $p\left(C_{i}\right)$ is a maximal
unlinked set in $G_{j_{1},j_{2}-1}$ of the form $V\times W$ with
$V$ a type. This implies that for every $u\in G_{j_{1},j_{2}-1}-p(C_{i})$
the set $p\left(C_{i}\right)$ contains both an element which is linked
with $u$ and an element which is unlinked with $u$ (see the construction
in the previous proof). But every element in $p\left(C_{j}\right)$
is linked with every element in $p\left(C_{i}\right)$ by $B_{j_{1},j_{2}-1}(p(u),p(v))=B_{0,1}(i,j)=1$,
which is a contradition. Hence $c_{i}<3^{j_{1}}2^{j_{2}-1}$. Thus
$\left|U\right|<3^{j_{1}}2^{j_{2}}$.

Now suppose $i$ and $j$ are unlinked. Then $\left|U\right|=3^{j_{1}}2^{j_{2}}$
if and only if $p\left(C_{i}\right)$ and $p\left(C_{j}\right)$ are
maximal unlinked sets. But these must also be unlinked with each other
and hence $p\left(C_{i}\right)=p\left(C_{j}\right)$ by maximality.
We have then that $p(U)=p(C_{i})=V\times W$ for $V$ a type and $W$
a maximal unlinked set. Thus $U=V\times\left(W\times\{i,j\}\right)$.

\textbf{Case 3: }Suppose at least three of the $c_{i}$ are non-zero.
Then for any unlinked pair of indices $\{i,j\}$, there exists $k$
such that $c_{k}\ne0$ and at least one of $\{i,k\},\{j,k\}$ are
unlinked. So it follows that $p$ is injective on $C_{i}\cup C_{j}$
so that $c_{i}+c_{j}\le3^{j_{1}}2^{j_{2}-1}$ with equality if and
only if $p(C_{i}\cup C_{j})=V\times W$ is an unlinked set of maximal
size in $G_{j_{1},j_{2}-1}$ by the inductive hypothesis. Then $|U|=(c_{0}+c_{2})+(c_{1}+c_{3})\le3^{j_{1}}2^{j_{2}}$
with equality if and only if $p(C_{0}\cup C_{2})$ and $p(C_{1}\cup C_{3})$
are unlinked of maximal size in $G_{j_{1},j_{2}-1}$. But we also
have that $|U|=(c_{0}+c_{3})+(c_{1}+c_{2})\le3^{j_{1}}2^{j_{2}}$
with equality if and only if $p(C_{0}\cup C_{3})$ and $p(C_{1}\cup C_{2})$
are unlinked of maximal size in $G_{j_{1},j_{2}-1}$. Suppose we have
equality, and by the inductive hypothesis suppose $p(C_{0}\cup C_{3})=V_{0,3}\times W_{0,3}$
and $p(C_{0}\cup C_{2})=V_{0,2}\cup W_{0,2}$. Then $p(C_{0})\subset p(C_{0}\cup C_{1})\cap p(C_{0}\cup C_{2})$,
so it follows that $V_{0,3}$ and $V_{0,2}$ must have the same type,
as theire intersection is non trivial and the types are mutually disjoint.
In particular $V_{0,3}=V_{0,2}$. By symmetry, we find that $V_{0,3}=V_{1,3}=V_{1,2}=V_{0,2}$.
In particular, since $p(U)=V_{0,3}\times W_{0,3}\cup V_{1,2}\times W_{1,2}$
we find that $p(U)=V\times W$ for $V=V_{0,3}$ a type and $W$ some
unlinked set. But $|W|=2^{j_{2}}$, and so it is a maximal unlinked
set in $Y_{2}^{j_{2}}$. 
\end{proof}
Combining this with the results of Section \ref{sec:Bounding-the-Error}
we have shown: 
\begin{prop}
For $j_{1}<k$ we have 
\[
S_{j_{1},j_{2},j_{3}}\left(X\right)=O\left(X\left(\log X\right)^{2^{j_{3}}a^{k}3^{k-j_{3}}-2^{j_{3}}a^{k}-1+\epsilon}\right).
\]
\end{prop}

\section{Computing the Moments }

It now remains to consider the contribution to $S_{k}\left(X\right)$
from $\sum_{\mathbf{A}}S_{k,0,0}\left(X,\mathbf{A}\right)$. Fix $\mathbf{A}$
as above. Then by $\left(\ref{eq:familycondition}\right)$ there will
be exactly $3^{k}$ unlinked variables $A_{u}$ greater than $X^{\ddagger}$
and all the remaining ones will satisfy $A_{v}=1$. This combined
with quadratic reciprocity reduces $S_{k,0,0}\left(X,\mathbf{A}\right)$
to the following expression, which we are further partitioning by
the congruence classes of each $D_{u}$:

\begin{eqnarray}
S_{k,0,0}\left(X,\mathbf{A}\right) & = & \sum_{\left(h_{u}\right)}\frac{1}{2^{4k}}\sum_{\left(D_{u}\right)}\mu^{2}\left(d\right)a^{k\omega(d)}\left[\prod_{u}\left(-1\right)^{\lambda_{k}(u)\left(u\right)\frac{h_{u}-1}{2}}\right]\label{eq:congruencepartition}\\
 &  & \times\left[\prod_{u,v}\left(-1\right)^{\Phi_{k}\left(u,v\right)\frac{h_{u}-1}{2}\frac{h_{v}-1}{2}}\right]+O\left(X\left(\log X\right)^{(3a)^{k}-a^{k}-1+\epsilon}\right)\nonumber 
\end{eqnarray}
where $h_{u}$ denotes the congruence class of $D_{u}$ mod $4$.
Next we would like to remove the congruence condition on the inner
sum over the $\left(D_{u}\right)$.

We will use the following result from \cite{fk1}, again refering
there for the proof. 
\begin{lem}[Fouvry-Klüners \cite{fk1}]
\label{lem:For-any-fixed}For any fixed tuple $\left(h_{u}\right)$
with $h_{u}\equiv\pm1\mod4$ we have 
\[
\sum_{\left(D_{u}\equiv h_{u}\mod4\right)}\mu^{2}\left(d\right)a^{k\omega\left(d\right)}=\frac{1}{2^{3^{k}}}\sum_{D_{u}}\mu^{2}\left(2d\right)a^{k\omega\left(d\right)}+O\left(X\left(\log X\right)^{-16^{k}\left(1+2^{k}\right)}\right).
\]
\end{lem}
Fix a maximal unlinked set of indices $\mathcal{U}$. We will call
any $\mathbf{A}$ satisfying $\left(\ref{eq:familycondition}\right)$
admissible for $\mathcal{U}$ if $A_{u}>X^{\ddagger}$ exactly when
$u\in\mathcal{U}$. Applying Lemma \ref{lem:For-any-fixed} to $\left(\ref{eq:congruencepartition}\right)$
and rearranging summations and also summing over all $\mathbf{A}$
admissible for $\mathcal{U}$ we get 
\begin{eqnarray*}
\sum_{\mathbf{A}\mathrm{\ admissible\ for\ }\mbox{\ensuremath{\mathcal{U}}}}S_{k,0,0}\left(X,\mathbf{A}\right) & = & \frac{1}{2^{4k+3^{k}}}\left(\sum_{\left(D_{u}\right)}\mu^{2}\left(2d\right)a^{k\omega\left(d\right)}\right)\sum_{\left(h_{u}\right)}\left[\prod_{u}\left(-1\right)^{\lambda_{k}\left(u\right)\frac{h_{u}-1}{2}}\right]\\
 &  & \times\left[\prod_{u,v}\left(-1\right)^{\Phi_{k}\left(u,v\right)\frac{h_{u}-1}{2}\frac{h_{v}-1}{2}}\right]+O\left(X\left(\log X\right)^{(3a)^{k}-a^{k}-1+\epsilon}\right).
\end{eqnarray*}
Note the summation is over $\left(D_{u}\right)$ such that there is
some $\mathbf{A}$ admissible for $\mathcal{U}$ with $A_{u}\le D_{u}\le\Delta A_{u}$.
However we can include the missing terms to extend the range to $1\le D_{u}<X$
at the cost an error of $O\left(X\left(\log X\right)^{(3a)^{k}-a^{k}-1+\epsilon}\right)$
by Lemma \ref{lem:-secondfamily}.

Summing over all maximal unlinked sets $\mathcal{U}$ we get 
\begin{eqnarray*}
S_{k,0,0}\left(X\right) & = & \frac{1}{2^{4k+3^{k}}}\left(\sum_{\mathcal{U}}\gamma\left(\mathcal{U}\right)\right)\left(\sum_{n<X}\mu^{2}\left(2n\right)\left(3a\right)^{k\omega(n)}\right)+O\left(X\left(\log X\right)^{(3a)^{k}-a^{k}-1+\epsilon}\right)
\end{eqnarray*}
where we define 
\[
\gamma\left(\mathcal{U}\right)=\sum_{\left(h_{u}\right)}\left[\prod_{u}\left(-1\right)^{\lambda_{k}\left(u\right)\frac{h_{u}-1}{2}}\right]\left[\prod_{u,v}\left(-1\right)^{\Phi_{k}\left(u,v\right)\frac{h_{u}-1}{2}\frac{h_{v}-1}{2}}\right].
\]
Recall (see $\left(\ref{eq:negative1mod4congruence1}\right)$ and
$\left(\ref{eq:negative1mod4congruence2}\right)$) that we are allowing
all possible congruence classes $\left(h_{u}\right)$ satisfying the
conditions: for all $1\le i\le k$ and $\left(u_{i},v_{i}\right)\in\left\{ \left(0,3\right),\left(2,5\right),\left(4,1\right)\right\} $

\begin{equation}
\prod_{u}h_{u}\prod_{v}h_{v}\equiv\begin{cases}
-1\mod4 & \mbox{if }\left(u_{i},v_{i}\right)=\left(0,3\right)\\
1\mod4 & \mbox{if }\left(u_{i},v_{i}\right)=\left(2,5\right),\left(4,1\right)
\end{cases}\label{eq:negative1mod4congruence_k}
\end{equation}
where the above products are over all $u$ with $u_{i}$ in the $i$th
position and all $v$ with $v_{i}$ in the $i$th position.

So far we have shown

\[
S_{k,0,0}\left(X\right)=\frac{1}{2^{4k+3^{k}}}\left(\sum_{\mathcal{U}}\gamma(U)\right)\left(\sum_{n<X}\mu^{2}(2n)(3a)^{k\omega(n)}\right)+O\left(X\left(\log X\right)^{(3a)^{k}-a^{k}-1+\epsilon}\right)
\]
We will now prove 
\begin{prop}
\label{prop:gamma1mod4negative} 
\[
\sum_{\mathcal{U}}\gamma\left(\mathcal{U}\right)=2^{3^{k}-k-1}.
\]
\end{prop}
Consider the vector space $\mathbb{F}_{2}^{3^{k}}$ where each coordinate
corresponds to an index in $\mathcal{U}$ using the lexicographic
order. We let each $y\in\mathbb{F}_{2}^{3^{k}}$ correspond to a tuple
of congruence classes $\left(h_{u}\right)$ by 
\begin{equation}
h_{u}\equiv-1\mod4\iff y_{u}=1.\label{eq:congruencetovector}
\end{equation}
Recursively define a matrix

\begin{align*}
M_{k} & =\begin{pmatrix}\vec{1} & \vec{0} & \vec{0}\\
\vec{0} & \vec{1} & \vec{0}\\
\vec{0} & \vec{0} & \vec{1}\\
M_{k-1} & M_{k-1} & M_{k-1}
\end{pmatrix}
\end{align*}
where $M_{1}=I_{3}$ is the identity matrix and $\vec{0},\vec{1}$
are row vectors of $0$s and $1$s respectively.

For example for $k=2$ we get 
\[
M_{2}=\left(\begin{array}{ccccccccc}
1 & 1 & 1\\
 &  &  & 1 & 1 & 1\\
 &  &  &  &  &  & 1 & 1 & 1\\
1 &  &  & 1 &  &  & 1\\
 & 1 &  &  & 1 &  &  & 1\\
 &  & 1 &  &  & 1 &  &  & 1
\end{array}\right).
\]
Then the $y\in\F_{2}^{3^{k}}$ satisfying condition $\left(\ref{eq:negative1mod4congruence_k}\right)$
are solutions of $M_{k}y=w$ for an appropriate $w\in\F_{2}^{3k}$.
This set of solutions is the coset $y+\ker M_{k}$.

Now we will prove \ref{prop:gamma1mod4negative} by combining the
following two lemmas. 
\begin{lem}
\label{lem:sumgamma} For all $k\ge1$ 
\[
\sum_{\mathcal{U}}\gamma\left(\mathcal{U}\right)=2^{k+\dim\ker M_{k}}.
\]
\end{lem}
\begin{proof}
Recall 
\[
\gamma(\mathcal{U})=\sum_{\left(h_{u}\right)}\left[\prod_{u}\left(-1\right)^{\lambda_{k}\left(u\right)\frac{h_{u}-1}{2}}\right]\left[\prod_{\{u,v\}}\left(-1\right)^{\Phi_{k}\left(u,v\right)\frac{h_{u}-1}{2}\frac{h_{v}-1}{2}}\right]
\]
By the discussion above, for any $\left(h_{u}\right)$ we let $x\in\F_{2}^{3^{k}}$
be the vector corresponding to it by $\left(\ref{eq:congruencetovector}\right)$.
Then $x$ belongs to a coset of $\ker M_{k}$, call it $y+\ker M_{k}$.
In particular, rephrasing the congruence conditions $\left(\ref{eq:negative1mod4congruence_k}\right)$
shows us that $\sum_{u:u_{i}=a}x_{u}=1$ if and only if $a\in\{0,3\}$.

Now consider that $\Phi(u,v)=0$ if $v\in A=\{1,3,5\}$. $\mathcal{U}$
is a largest maximal unlinked set, and so has a type $s\in S$, so
it follows that 
\[
\Phi_{k}(u,v)=\sum_{i}\Phi(u_{i},v_{i})=\sum_{i:s_{i}=B}\Phi(u_{i},v_{i})
\]
And similarly, $\lambda|_{A}=0$ so that 
\[
\lambda(u)=\sum_{i}\lambda(u_{i})=\sum_{i:s_{i}=B}\lambda(u_{i})
\]
This way we show that 
\begin{align*}
\sum_{\mathcal{U}}\gamma(\mathcal{U}) & =\sum_{s\in S}\sum_{x\in y+\ker M_{k}}\left[\prod_{u}\left(-1\right)^{\sum_{i:s_{i}=B}\lambda(u_{i})x_{u}}\right]\left[\prod_{\{u,v\}}\left(-1\right)^{\sum_{s_{i}=B}\Phi(u_{i},v_{i})x_{u}x_{v}}\right]\\
 & =\sum_{s\in S}\sum_{x\in y+\ker M_{k}}(-1)^{\sum_{u}\sum_{i:s_{i}=B}\lambda(u_{i})x_{u}+\sum_{\{u,v\}}\sum_{i:s_{i}=B}\Phi(u_{i},v_{i})x_{u}x_{v}}.
\end{align*}
Call $\mathcal{U}_{B}=B^{k}$. Interchanging summations we can apply
the binomial theorem to the sum over types $s\in S$ to show 
\[
\sum_{\mathcal{U}}\gamma(\mathcal{U})=\sum_{x\in y+\ker M_{k}}\prod_{j=1}^{k}\left(1+\left(-1\right){}^{\sum_{u\in\mathcal{U}_{B}}\lambda(u_{j})x_{u}+\sum_{\{u,v\}\subset\mathcal{U}_{B}}\Phi(u_{j},v_{j})x_{u}x_{v}}\right).
\]
Notice that for all $j=1,...,k$ we have 
\begin{align*}
\sum_{\{u,v\}\subset\mathcal{U}_{B}}\Phi(u_{j},v_{j})x_{u}x_{v} & =\sum_{u_{j}=0,v_{j}=2}x_{u}x_{v}+\sum_{u_{j}=0,v_{j}=4}x_{u}x_{v}+\sum_{u_{j}=2,v_{j}=4}x_{u}x_{v}\\
 & =\sum_{u_{j}=0}x_{u}\sum_{v_{j}=2}x_{v}+\sum_{u_{j}=0}x_{u}\sum_{v_{j}=4}x_{v}+\sum_{u_{j}=2}x_{u}\sum_{v_{j}=4}x_{v}
\end{align*}
and 
\[
\sum_{u\in\mathcal{U}_{B}}\lambda(u_{j})x_{u}=\sum_{u_{j}=2}x_{u}+\sum_{u_{j}=4}x_{u}
\]
and it follows that these are 0 from the conditions $\left(\ref{eq:negative1mod4congruence_k}\right)$.
Thus 
\begin{align*}
\sum_{\mathcal{U}}\gamma(\mathcal{U}) & =\sum_{x\in y+\ker M_{k}}\prod_{j=1}^{k}2\\
 & =2^{k+\dim\ker M_{k}}.
\end{align*}
\end{proof}
\begin{lem}
\label{thm:kerMk} $\dim\ker M_{k}=3^{k}-2k-1$. 
\end{lem}
\begin{proof}
Without loss of generality suppose $U$ has type $s$ with $s_{1}=\{0,2,4\}$.
For $x\in\F_{2}^{3^{k}}$ and $j\in s_{1}$ let $p_{j}\left(x\right)$
be the projection onto $\F_{2}^{3^{k-1}}$ of the coordinates $x_{u}$
of $x$ where $u_{1}=j$.

Now $x\in\ker M_{k}$ if and only if $\sum_{j\in s_{1}}p_{j}\left(x\right)\in\ker M_{k-1}$
and $\alpha_{k-1}(p\left(x_{j}\right))=0$ for all $j\in s_{1}$,
with $\alpha_{k-1}:\F_{2}^{3^{k-1}}\rightarrow\F_{2}$ the augmentation
map defined by $v\mapsto\textbf{1}\cdot v$.

It is clear that $\ker M_{k-1}\subset\ker\alpha_{k-1}$. So we have
that 
\[
\ker M_{k}=\left\{ x\mid\sum_{j\in s_{1}}p_{j}\left(x\right)\in\ker M_{k-1},p\left(x_{j}\right)\in\ker\alpha_{k-1}\right\} .
\]
There are $|\ker\alpha_{3^{k}-1}|$ choices for $p_{0}\left(x\right)$
and $p_{2}\left(x\right)$. Then we have 
\[
p_{4}\left(x\right)\in\left(p_{0}\left(x\right)+p_{2}\left(x\right)+\ker M_{k-1}\right)\cap\ker\alpha_{k-1}.
\]
That is $p_{4}\left(x\right)$ belongs to a coset of $\ker M_{k-1}\subset\ker\alpha_{k-1}$,
so there are $|\ker M_{k-1}|$ choices. So we have 
\[
|\ker M_{k}|=|\ker\alpha_{k-1}|^{2}|\ker M_{k-1}|=2^{2(3^{k-1}-1)}|\ker M_{k-1}|.
\]
Clearly $M_{1}$ is the identity map, and so $\ker M_{1}=0$. Then
a simple induction shows that $|\ker M_{k}|=2^{3^{k}-2k-1}$ which
completes the proof. 
\end{proof}
Combining Lemma \ref{lem:sumgamma} and Lemma \ref{thm:kerMk} immediately
proves Proposition \ref{prop:gamma1mod4negative}.

In summary we have shown that 
\begin{equation}
S_{k}\left(X\right)=\frac{1}{2^{5k+1}}\left(\sum_{n<X}\mu^{2}(2n)(3a)^{k\omega(n)}\right)+O\left(X\left(\log X\right)^{(3a)^{k}-a^{k}-1+\epsilon}\right)
\end{equation}
and the first case of Theorem \ref{thm:mainthm3} follows by noting
that 
\begin{align*}
\sum_{n<X}\mu^{2}(2n)(3a)^{k\omega(n)} & =2\sum_{d\in\mathcal{D}_{X,1,4}^{-}}(3a)^{k\omega(d)}+o(X).
\end{align*}

\section{Main results and implications}

We now have the tools to prove Theorems \ref{thm:mainthm-1} and \ref{thm:mainthm2-1}
as corollaries to Theorem \ref{thm:mainthm3}. We restate the theorems
here followed by their proofs: 
\begin{thm}
\label{thm:mainthm-1} Let $\left(G',G\right)=\left(H_{8}\rtimes C_{2},H_{8}\right)$.
For a quadratic field $K$ with discriminant $d$ let $f\left(d\right)$
be the number of unramified everywhere $\left(G',G\right)$-extensions
of $K$. Let $g\left(d\right)=3^{\omega\left(d\right)}$. Then for
all $k\in\mathbb{Z}_{\ge1}$, 
\[
\mathbf{E}^{-}\left(\left(f\left(d\right)/g\left(d\right)\right)^{k}\right)=\left(\frac{1}{32}\right)^{k}
\]
and 
\[
\mathbf{E}^{+}\left(\left(f\left(d\right)/g\left(d\right)\right)^{k}\right)=\left(\frac{1}{192}\right)^{k}
\]
Thus the function $f\left(d\right)/g\left(d\right)$ determines the
point mass distribution at $1/32$ (resp. $1/192$) on $\mathbb{R}$. 
\end{thm}
\begin{proof}
This follows immediately from setting $a=1/3$ in Theorem \ref{thm:mainthm3}
(and multiplying by $3^{-k}$ for the even cases). 
\end{proof}
Using this theorem we can show the sequence $f(d)/g(d)$ determines
a distribution in the following sense. Clearly $\mu_{n}\left(U\right)=\frac{1}{\left|\mathcal{D}_{n}^{\pm}\right|}\left|\left\{ f(d)/g(d)\in U\mid d<n\right\} \right|$
is a probability measure on $\mathbb{R}$ for all $n$. Let $c=1/32$.
By Theorem \ref{thm:mainthm-1} $\lim_{n\longrightarrow\infty}\mathbb{E}_{\mu_{n}}\left(x^{k}\right)=c^{k}$
for all $k$. Let $\mu_{c}$ be the point-mass at $c$. The measure
$\mu_{c}$ has moments $c^{k}$ and is determined by its moments.
Then it is a standard fact (see for instance \cite{Billingsley})
that the $\mu_{n}$ converge to $\mu_{c}$ in distribution, as $n\longrightarrow\infty$.

Recall that $H_{8}^{k}\rtimes_{\sigma}C_{2}$ denotes the group where
the action of $\sigma$ on each coordinate gives $H_{8}\rtimes C_{2}$
according to our definition. 
\begin{thm}
\label{thm:mainthm2-1}Let $k\in\mathbb{Z}_{\ge1}$, $G=H_{8}^{k}$,
and $G'=H_{8}^{k}\rtimes_{\sigma}C_{2}$. Define $\Surj_{\sigma}\left(\mathrm{Gal}\left(K^{un}/K\right),G\right)$
be the set of surjections which lift to a surjection $\mathrm{Gal}\left(K^{un}/\Q\right)\rightarrow G'$.
Then 
\[
\sum_{d\in\mathcal{D}_{X}^{-}<X}\left|\Surj_{\sigma}\left(\mathrm{Gal}\left(K^{un}/K\right),G\right)\right|=\left(\frac{1}{4}\right)^{k}\left(\sum_{d\in\mathcal{D}_{X}^{-}}3^{k\omega(d)}\right)+O\left(X\left(\log X\right)^{3^{k}-2+\epsilon}\right)
\]
and 
\[
\sum_{d\in\mathcal{D}_{X}^{+}}\left|\Surj_{\sigma}\left(\mathrm{Gal}\left(K^{un}/K\right),G\right)\right|=\left(\frac{1}{24}\right)^{k}\left(\sum_{d\in\mathcal{D}_{X}^{+}}3^{k\omega(d)}\right)+O\left(X\left(\log X\right)^{3^{k}-2+\epsilon}\right)
\]
\end{thm}
\begin{proof}
As a notational convenience, let $G_{K}^{un}=\mathrm{Gal}\left(K^{un}/K\right)$
throught the proof. Clearly taking $a=1$ and noting that $\#\Surj_{\sigma}\left(G_{d}^{ur},H_{8}\right)=8f(d)$
we see that 
\[
\sum_{|d|<X}\left|\Surj_{\sigma}\left(G_{d}^{ur},H_{8}\right)^{k}\right|=\begin{cases}
\left(\frac{1}{4}\right)^{k}\left(\sum_{\text{odd }|d|<X}3^{k\omega(d)}\right)+O\left(X\left(\log X\right)^{3^{k}-2+\epsilon}\right) & \text{imaginary}\\
\left(\frac{1}{24}\right)^{k}\left(\sum_{\text{odd }|d|<X}3^{k\omega(d)}\right)+O\left(X\left(\log X\right)^{3^{k}-2+\epsilon}\right) & \text{real}
\end{cases}
\]
Consider that we have 
\[
\left|\Surj_{\sigma}\left(G_{d}^{ur},H_{8}^{k}\right)\right|=\sum_{A\le H_{8}^{k}}\mu_{H_{8}^{k}}(A)\left|\Hom_{\sigma}\left(G_{d}^{ur},A\right)\right|
\]
as in \cite{hall}. A simple exercise in group theory shows that all
subgroups of $H_{8}^{k}$ are isomorphic to $H_{8}^{j_{1}}\times C_{4}^{j_{2}}\times C_{2}^{j_{3}}$
for some $j_{1}+j_{2}+j_{3}\le k$. So we have that 
\begin{eqnarray*}
\left|\Surj_{\sigma}\left(G_{d}^{ur},H_{8}^{k}\right)\right|=\sum_{j_{1}+j_{2}+j_{2}\le k} &  & \mu_{H_{8}^{k}}\left(H_{8}^{j_{1}}\times C_{4}^{j_{2}}\times C_{2}^{j_{3}}\right)\\
 &  & \times\left|\Hom_{\sigma}\left(G_{d}^{ur},H_{8}\right)\right|^{j_{1}}\left|\Hom_{\sigma}\left(G_{d}^{ur},C_{4}\right)\right|^{j_{2}}\left|\Hom_{\sigma}\left(G_{d}^{ur},C_{4}\right)\right|^{j_{3}}.
\end{eqnarray*}
A similar argument shows that the $j_{1}=k$ term is the main term,
and then 
\begin{eqnarray*}
\sum_{|d|<X}\left|\Surj_{\sigma}\left(G_{d}^{ur},H_{8}^{k}\right)\right| & = & \sum_{|d|<X}\left|\Hom_{\sigma}\left(G_{d}^{ur},H_{8}\right)\right|^{k}+O\left(X\left(\log X\right)^{3^{k}-2+\epsilon}\right)\\
 & = & \sum_{|d|<X}\left(\left|\Surj_{\sigma}\left(G_{d}^{ur},H_{8}\right)\right|-3\left|\Surj_{\sigma}\left(G_{d}^{ur},C_{4}\right)\right|+2\left|\Surj_{\sigma}\left(G_{d}^{ur},C_{2}\right)\right|\right)^{k}\\
 &  & +O\left(X\left(\log X\right)^{3^{k}-2+\epsilon}\right)\\
 & = & \sum_{|d|<X}\sum_{i_{1}+i_{2}+i_{3}}\binom{k}{i_{1},i_{2},i_{3}}\left|\Surj_{\sigma}\left(G_{d}^{ur},H_{8}\right)\right|^{i_{1}}\\
 &  & \hspace{2cm}\times3^{i_{2}}\left|\Surj_{\sigma}\left(G_{d}^{ur},C_{4}\right)\right|^{i_{2}}2^{i_{3}}\left|\Surj_{\sigma}\left(G_{d}^{ur},C_{2}\right)\right|^{i_{3}}\\
 &  & +O\left(X\left(\log X\right)^{3^{k}-2+\epsilon}\right)
\end{eqnarray*}
Where again a similar argument shows the $i_{1}=k$ term is the main
term, so that 
\begin{eqnarray*}
\sum_{|d|<X}\left|\Surj_{\sigma}\left(G_{d}^{ur},H_{8}^{k}\right)\right| & = & \sum_{|d|<X}\left|\Surj_{\sigma}\left(G_{d}^{ur},H_{8}\right)\right|^{k}+O\left(X\left(\log X\right)^{3^{k}-2+\epsilon}\right).
\end{eqnarray*}
\end{proof}
This additionally answers a conjecture of Wood \cite{Wood} for nonabelian
Cohen-Lenstra heuristcs in the case of $G=H_{8}^{k}$ and $[G':G]=2$,
which says we expect the sum in this corollary to be $\gg X$. 
\begin{cor}
\label{cor:The-density-of-2}The density of quadratic fields $K$
with $\mathrm{Gal}\left(K^{un}/K\right)=H_{8}^{m}$ is equal to $0$
for any positive $m\in\mathbb{Z}$. 
\end{cor}
\begin{proof}
Repeating the proof of Theorem \ref{thm:mainthm2-1} with $a=1/3$
gives 
\begin{eqnarray*}
\sum_{|d|<X}\frac{\left|\Surj_{\sigma}\left(G_{d}^{ur},H_{8}^{m}\right)\right|}{g\left(d\right)^{m}} & = & \sum_{|d|<X}\left(\frac{\left|\Surj_{\sigma}\left(G_{d}^{ur},H_{8}\right)\right|}{g\left(d\right)}\right)^{m}+o\left(X\right).
\end{eqnarray*}
\[
\]
for any $m$, from which it is clear that the $k$th moments of the
function $\left|\Surj_{\sigma}\left(G_{d}^{ur},H_{8}^{m}\right)\right|/g\left(d\right)^{m}$
will be $m$th powers of the $k$th moments of $\left|\Surj_{\sigma}\left(G_{d}^{ur},H_{8}\right)\right|/g\left(d\right)$,
which are given by Theorem \ref{thm:mainthm-1}. Thus the distribution
of the values of $\left|\Surj_{\sigma}\left(G_{d}^{ur},H_{8}^{m}\right)\right|/g\left(d\right)^{m}$
will again by a point mass supported at some positive real number
$c$. By definition this means that for any fixed $m\in\mathbb{Z}$
and $\epsilon>0$, one hundred percent of quadratic fields satisfy
$\left|\left|\Surj_{\sigma}\left(G_{d}^{ur},H_{8}^{m}\right)\right|/g\left(d\right)^{m}-c\right|<\epsilon$.
In particular $\left|\Surj_{\sigma}\left(G_{d}^{ur},H_{8}^{m}\right)\right|$
is non-zero one hundred percent of the time, which means there is
at least one $H_{8}^{m}$ extension. But this holds for any $m$.
The corollary follows. 
\end{proof}

\section{\label{sec:Future-Work}Future work and conjectures}

As mentioned in the introduction \cite{lemmermeyer2groups} contains
conditions for the existence of unramified $\left(G',G\right)$-extensions
in several other cases. In the case of $\left(G',G\right)=(D_{4}\times C_{2},D_{4})$
a proof similar to Proposition \ref{prop:The-number-of} gives the
formula

\begin{eqnarray*}
f_{(D_{4}\times C_{2},D_{4})}(d) & = & \frac{1}{2}\sum_{d=d_{1}d_{2}d_{3}}\frac{\prod_{i}2^{\omega(d_{i})-1}}{2^{\omega(d)}}2^{\omega\left(d_{3}\right)}\prod_{p\mid d_{1}}\left(1+\left(\frac{d_{2}}{p}\right)\right)\prod_{p\mid d_{2}}\left(1+\left(\frac{d_{1}}{p}\right)\right).
\end{eqnarray*}

Our choice of normalizing constant in Conjecture \ref{conj:For-a-quadratic}
is based on the following simple heuristic. Consider the case $(D_{4}\times C_{2},D_{4})$.
If we assume each residue symbol takes values $\pm1$ with probability
$1/2$ then the expected value of the product is 
\[
\left(\frac{1}{2}\right)^{\omega\left(d_{1}d_{2}\right)}\cdot2^{\omega\left(d_{1}d_{2}\right)}+\left(1-\left(\frac{1}{2}\right)^{\omega\left(d\right)}\right)\cdot0=1
\]
Hence on average one expects $f$ to be on the order of 
\begin{eqnarray*}
\sum_{d=d_{1}d_{2}d_{3}}2^{\omega\left(d_{3}\right)} & = & \sum_{i=0}^{\omega\left(d\right)}2^{\omega\left(i\right)}2^{\omega\left(d\right)-i}\binom{\omega\left(d\right)}{i}\\
 & = & 4^{\omega\left(d\right)}.
\end{eqnarray*}
The heuristic applies similarly to all such expressions.

\section{Appendix}

We maintain all the notation from the previous sections, notably the
matrix $M_{k}$.

\subsection{The case $d<0$ and $d\equiv4\mod8$}

Next we consider fundamental discriminants $d<0$ and $d\equiv4$
mod $8$. The number of $H_{8}$ extensions of a quadratic field $k$
Galois over $\Q$ with such a discriminant is 
\[
f\left(d\right)=f_{1}\left(d\right)-f_{2}\left(d\right)+2^{\omega(d)-4}
\]
where we define

\begin{eqnarray*}
f_{1}\left(d\right) & = & \frac{1}{2}\sum_{d=d_{1}d_{2}d_{3}}\left(1+\Leg{d_{2}d_{3}}{2}\right)\frac{\prod_{i}2^{\omega\left(d_{i}\right)-1}}{2^{\omega\left(d\right)}}\prod_{p\mid d_{3}}\left(1+\left(\frac{-d_{1}d_{2}}{p}\right)\right)\\
 &  & \times\prod_{p\mid d_{1}}\left(1+\left(\frac{d_{2}d_{3}}{p}\right)\right)\prod_{p\mid d_{2}}\left(1+\left(\frac{-d_{3}d_{1}}{p}\right)\right)\\
 &  & +\sum_{d=d_{1}d_{2}d_{3}}\left(1+\Leg{-d_{2}d_{3}}{2}\right)\frac{\prod_{i}2^{\omega\left(d_{i}\right)-1}}{2^{\omega\left(d\right)}}\prod_{p\mid d_{3}}\left(1+\left(\frac{-d_{1}d_{2}}{p}\right)\right)\\
 &  & \times\prod_{p\mid d_{1}}\left(1+\left(\frac{-d_{2}d_{3}}{p}\right)\right)\prod_{p\mid d_{2}}\left(1+\left(\frac{d_{3}d_{1}}{p}\right)\right)\\
\end{eqnarray*}
and 
\begin{eqnarray*}
f_{2}\left(d\right) & = & \sum_{d=d_{1}d_{2}}\frac{\prod_{i}2^{\omega\left(d_{i}\right)-1}}{2^{\omega\left(d\right)}}\left(1+\Leg{d_{2}}{2}\right)\prod_{p\mid d_{1}}\left(1+\left(\frac{d_{2}}{p}\right)\right)\prod_{p\mid d_{2}}\left(1+\left(\frac{-d_{1}}{p}\right)\right)\\
 &  & +\sum_{d=d_{1}d_{2}}\frac{\prod_{i}2^{\omega\left(d_{i}\right)-1}}{2^{\omega\left(d\right)}}\left(1+\Leg{-d_{2}}{2}\right)\prod_{p\mid d_{1}}\left(1+\left(\frac{-d_{2}}{p}\right)\right)\prod_{p\mid d_{2}}\left(1+\left(\frac{-d_{1}}{p}\right)\right)\\
\end{eqnarray*}
where the factoriation $d=d_{1}d_{2}d_{3}$ is into integers such
that $d_{1}\equiv-1\mod4$ and $d_{i}\equiv1\mod4$, and each sum
corresponds to $d_{1}<0$ and $d_{i}<0$ for $i\ne1$. Otherwise this
follows from the same reasoning as in the previous subsections. As
before this gives 
\begin{eqnarray*}
f_{1}\left(d\right) & = & \frac{1}{2^{3}}\sum_{m=0}^{1}\sum_{d=-4D_{0}D_{1}D_{2}D_{3}D_{4}D_{5}}\left(1+\Leg{D_{1}D_{2}D_{4}D_{5}}{2}\right)\left(\frac{1}{2}\Leg{-1}{D_{2}D_{4}}m+\Leg{-1}{D_{0}D_{4}}(1-m)\right)\\
 &  & \times\left(\frac{D_{0}D_{3}D_{2}D_{5}}{D_{4}}\right)\left(\frac{D_{2}D_{5}D_{4}D_{1}}{D_{0}}\right)\left(\frac{D_{4}D_{1}D_{0}D_{3}}{D_{2}}\right)
\end{eqnarray*}
and 
\begin{eqnarray*}
f_{2}\left(d\right)=\frac{1}{2^{3}}\sum_{m=0}^{1}\sum_{d=-4E_{0}E_{1}E_{2}E_{3}}\left(1+\Leg{E_{2}E_{3}}{2}\right)\left(\Leg{-1}{E_{2}}m+\Leg{-1}{E_{0}}(1-m)\right)\left(\frac{E_{2}E_{3}}{E_{0}}\right)\left(\frac{E_{0}E_{1}}{E_{2}}\right).
\end{eqnarray*}
where the sum is over factorizations which satisfy the congruences
\begin{eqnarray*}
D_{0}D_{3} & \equiv & (-1)^{m+1}\mod4\\
D_{2}D_{5} & \equiv & (-1)^{m+1}\mod4\\
D_{4}D_{1} & \equiv & 1\mod4
\end{eqnarray*}
\begin{eqnarray*}
E_{0}E_{1} & \equiv & (-1)^{m+1}\mod4\\
E_{2}E_{3} & \equiv & (-1)^{m+1}\mod4
\end{eqnarray*}
As before we compute 
\begin{eqnarray*}
f_{j_{1},j_{2}}\left(d\right) & = & \frac{1}{2^{3j_{1}+3j_{2}}}\sum_{\left(D_{u}\right)}\sum_{C\subset\{1,...,j_{1}+j_{2}\}}\prod_{u}\Leg{D_{u}}{2}^{\sum_{i\in C}Q(u_{i})}\sum_{J_{1}\subset\{1,...,j_{1}\}}\sum_{J_{2}\subset\{j_{1}+1,...,j_{1}+j_{2}\}}\frac{1}{2^{|J_{1}|}}\\
 &  & \times\Leg{-1}{D_{u}}^{\sum_{i=1}^{j_{1}+j_{2}}\lambda^{1}(u_{i})\chi_{J_{1}}(i)+\lambda^{2}(u_{i})(1-\chi_{J_{1}}(u_{i}))+\gamma^{1}(u_{i})\chi_{J_{2}}(i)+\gamma^{2}(u_{i})(1-\chi_{J_{2}}(i))}\\
 &  & \times\prod_{u,v}\left(\frac{D_{u}}{D_{v}}\right)^{\Phi_{j_{1}}\left(u,v\right)+\Psi_{j_{2}}\left(u,v\right)}
\end{eqnarray*}
For $Q(u)=0$ if $u\in\{0,3\}$ and $1$ otherwise, $\lambda^{1}(u)=1$
iff $u=2,4$, $\lambda^{2}(u)=1$ iff $u=0,4$, $\gamma^{1}(u)=1$
iff $u=2$, and $\gamma^{2}(u)=1$ iff $u=0$. The sum is over $6^{j_{1}}4^{j_{2}}$
tuples of integers $\left(D_{u}\right)$ which satisfy $\prod_{u}D_{u}=d$
and the congruence conditions: for all $1\le i\le j_{1}$ and $\left(u_{i},v_{i}\right)\in\left\{ \left(0,3\right),\left(2,5\right),\left(4,1\right)\right\} $
and all $j_{1}+1\le i\le j_{1}+j_{2}$ and $\left(u_{i},v_{i}\right)\in\left\{ \left(0,1\right),\left(2,3\right)\right\} $

\begin{eqnarray}
\prod_{u}D_{u}\prod_{v}D_{v} & \equiv & \begin{cases}
(-1)^{|J_{1}|+1}\mod4 & (u_{i},v_{i})=(0,3),(2,5),i\le j_{1}\\
(-1)^{|J_{2}|+1}\mod4 & i\ge j_{2}\\
1\mod4 & \text{else}
\end{cases}
\end{eqnarray}
where the above products are over all $u$ with $u_{i}$ in the $i$th
position and all $v$ with $v_{i}$ in the $i$th position.

Thus multiplying by $2^{kj_{3}\omega(d)}a^{k\omega\left(d\right)}$
and summing over discriminants $d<0$ with $d\equiv1\mod4$ we get
\begin{eqnarray*}
\sum_{d<X}2^{j_{3}(\omega(d))}a^{\omega\left(d\right)}f_{j_{1},j_{2}}\left(d\right) & = & \frac{1}{3^{j_{1}}2^{3j_{1}+3j_{2}}}\sum_{\left(D_{u}\right)}\mu^{2}\left(\prod_{u}D_{u}\right)a^{k\omega\left(d\right)}2^{kj_{3}\omega(d)}\\
 &  & \times\sum_{C\subset\{1,...,j_{1}+j_{2}\}}\sum_{J_{1}\subset\{1,...,j_{1}\}}\sum_{J_{2}\subset\{j_{1}+1,...,j_{1}+j_{2}\}}\prod_{u}\Leg{D_{u}}{2}^{\sum_{i\in C}Q(u_{i})}\\
 &  & \times\frac{1}{2^{|J_{1}|}}\Leg{-1}{D_{u}}^{\sum_{i=1}^{j_{1}+j_{2}}\lambda^{1}(u_{i})\chi_{J_{1}}(i)+\lambda^{2}(u_{i})(1-\chi_{J_{1}}(u_{i}))+\gamma^{1}(u_{i})\chi_{J_{2}}(i)+\gamma^{2}(u_{i})(1-\chi_{J_{2}}(i))}\\
 &  & \times\prod_{u,v}\left(\frac{D_{u}}{D_{v}}\right)^{\Phi_{j_{1}}\left(u,v\right)+\Psi_{j_{2}}\left(u,v\right)}
\end{eqnarray*}
where the sum is over $6^{j_{1}}4^{j_{2}}$ tuples of integers $\left(D_{u}\right)$
which satisfy $\prod_{u}D_{u}<X$ and the conditions $\left(\ref{eq:congruence2}\right)$.

In this case, we have 
\begin{eqnarray*}
\sum_{\mathbf{A}\mathrm{\ admissible\ for\ }\mbox{\ensuremath{\mathcal{U}}}}S_{k,0,0}\left(X,\mathbf{A}\right) & = & \frac{1}{2^{3k}}\left(\sum_{\left(D_{u}\right)}\mu^{2}\left(d/2\right)a^{k\omega\left(d/4\right)}\right)\\
 &  & \times\sum_{\left(h_{u}\right)}\left[\sum_{C\subset\{1,...,k\}}\prod_{u}\left(-1\right)^{\sum_{i}Q(u_{i})\chi_{C}(i)\frac{h_{u}^{2}-1}{8}}\right]\\
 &  & \sum_{J\subset\{1,...,k\}}\frac{1}{2^{|J|}}\prod_{u}(-1)^{\frac{h_{u}-1}{2}\sum_{i=1}^{k}\lambda^{1}(u_{i})\chi_{J}(i)+\lambda^{2}(u_{i})(1-\chi_{J}(u_{i}))}\\
 &  & \times\left[\prod_{u,v}\left(-1\right)^{\Phi_{k}\left(u,v\right)\frac{h_{u}-1}{2}\frac{h_{v}-1}{2}}\right]+O\left(X\left(\log X\right)^{(3a)^{k}-a^{k}-1+\epsilon}\right).
\end{eqnarray*}

For $h_{u}$ the congruence class of $D_{u}\mod8$ and $d=4\prod D_{u}$
and we grouped the $4$ factor with the first discriminant in the
factorization (i.e. for $k=1$ the factorization is $d=(4D_{0}D_{3})(D_{1}D_{4})(D_{2}D_{5})$).

Then, noting that $h_{u}$ is one out of four choices for odd numbers
$\mod8$, we get 
\begin{eqnarray*}
S_{k,0,0}(X) & = & \frac{1}{2^{3k}2^{2\cdot3^{k}}}\left(\sum_{\mathcal{U}}\gamma(\mathcal{U})\right)\left(\sum_{4n<X}\mu^{2}(2n)(3a)^{k\omega(n)}\right)+O\left(X(\log X)^{(3a)^{k}-a^{k}-1+\epsilon}\right)
\end{eqnarray*}
where we define 
\begin{eqnarray*}
\gamma(\mathcal{U}) & = & \sum_{\left(h_{u}\right)}\left[\sum_{C\subset\{1,...,k\}}\prod_{u}\left(-1\right)^{\sum_{i}Q(u_{i})\chi_{C}(i)\frac{h_{u}^{2}-1}{8}}\right]\left[\prod_{u,v}\left(-1\right)^{\Phi_{k}\left(u,v\right)\frac{h_{u}-1}{2}\frac{h_{v}-1}{2}}\right]\\
 &  & \times\left[\sum_{J\subset\{1,...,k\}}\frac{1}{2^{|J|}}\prod_{u}(-1)^{\frac{h_{u}-1}{2}\sum_{i=1}^{k}\lambda^{1}(u_{i})\chi_{J}(i)+\lambda^{2}(u_{i})(1-\chi_{J}(u_{i}))}\right]
\end{eqnarray*}
allowing odd congruence classes $h_{u}\mod8$ satisfying the following
conditions: for all $1\le i\le k$ 
\begin{eqnarray}
\prod_{u}h_{u}\prod_{v}h_{v}\equiv\begin{cases}
1\mod4 & (u_{i},v_{i})\in\{(1,4)\}\\
(-1)^{|J|+1}\mod4 & (u_{i},v_{i})\in\{(0,3),(2,5)\}
\end{cases}
\end{eqnarray}
where the above products are taken over all $u$ with $u_{i}$ in
the $i$th position and $v$ with $v_{i}$ in the $i$th position.
(Recalling that the odd parts of even discriminants are $-1\mod4$
if $8\nmid d$). If we call $x\equiv\left(\frac{h_{u}-1}{2}\right)\mod2$
an element of $\F_{2}^{3^{k}}$, then $x\in y+\ker M_{k}$ a coset
of $\ker M_{k}$.

Now consider that $\Phi(u,v)=0$ if $u,v\in A=\{1,3,5\}$. $\mathcal{U}$
is a largest maximal unlinked set, and so has a type $s\in S$, so
it follows that 
\[
\Phi_{k}(u,v)=\sum_{i}\Phi(u_{i},v_{i})=\sum_{i:s_{i}=B}\Phi(u_{i},v_{i})
\]
This way we show that 
\begin{eqnarray*}
\sum_{\mathcal{U}}\gamma(\mathcal{U}) & = & \sum_{s\in S}\sum_{x\in y+\ker M_{k}}\left[\sum_{(h_{u}):\frac{h_{u}-1}{2}\equiv x_{u}\mod2}\sum_{C\subset\{1,...,k\}}\prod_{u}\left(-1\right)^{\sum_{i\in C}Q(u_{i})\left(\frac{h_{u}^{2}-1}{8}\right)}\right]\\
 &  & \times\left[\prod_{\{u,v\}}\left(-1\right)^{\sum_{s_{i}=B}\Phi(u_{i},v_{i})x_{u}x_{v}}\right]\\
 &  & \times\left[\sum_{J\subset\{1,...,k\}}\frac{1}{2^{|J|}}\prod_{u}(-1)^{x_{u}\sum_{i=1}^{k}\lambda^{1}(u_{i})\chi_{J}(i)+\lambda^{2}(u_{i})(1-\chi_{J}(u_{i}))}\right]
\end{eqnarray*}
Notice that for each $x_{u}$, there are two choices of $h_{u}\mod8$
such that $\frac{h_{u}-1}{2}\equiv x_{u}\mod2$, because $h_{u}$
and $5h_{u}$ give the same image. If we fix an $(h(x)_{u})$ satisfying
this property without loss of generality also satisfying $\frac{h_{u}^{2}-1}{8}\equiv0\mod2$,
then we have 
\begin{align*}
\sum_{(h_{u}):\frac{h_{u}-1}{2}\equiv x_{u}\mod2}\sum_{C\subset\{1,...,k\}}\prod_{u}\left(-1\right)^{\sum_{i\in C}Q(u_{i})\left(\frac{h_{u}^{2}-1}{8}\right)} & =\sum_{T\subset\mathcal{U}}\sum_{C\subset\{1,...,k\}}\prod_{u}\left(-1\right)^{\sum_{i\in C}Q(u_{i})\left(\frac{(h_{u}5^{\chi_{T}(u)})^{2}-1}{8}\right)}\\
 & =\sum_{T\subset\mathcal{U}}\sum_{C\subset\{1,...,k\}}\prod_{u}\left(-1\right)^{\sum_{i\in C}Q(u_{i})\left(\frac{5^{2\chi_{T}(u)}-1}{8}\right)}
\end{align*}
Now, $\frac{5^{2}-1}{8}\equiv1\mod2$. So it follows that $\left(\frac{5^{2\chi_{T}(u)}-1}{8}\right)\equiv(\chi_{T}(u))\mod2$
so we have 
\begin{align*}
\sum_{(h_{u}):\frac{h_{u}-1}{2}\equiv x_{u}\mod2}\sum_{C\subset\{1,...,k\}}\prod_{u}\left(-1\right)^{\sum_{i\in C}Q(u_{i})\left(\frac{h_{u}^{2}-1}{8}\right)} & =\sum_{T\subset\mathcal{U}}\sum_{C\subset\{1,...,k\}}\prod_{u}\left(-1\right)^{\sum_{i\in C}Q(u_{i})\chi_{T}(u)}\\
 & =\sum_{T\subset\mathcal{U}}\sum_{C\subset\{1,...,k\}}\left(-1\right)^{\sum_{u,i}Q(u_{i})\chi_{C}(i)\chi_{T}(u)}\\
 & =\sum_{C\subset\{1,...,k\}}\prod_{u}\left(1+(-1)^{\sum_{i}Q(u_{i})\chi_{C}(i)}\right)
\end{align*}
By using the binomial theorem in two different ways. Now, for $s_{i}$
being $A$ or $B$, we can find a $u\in\mathcal{U}$ such that $\sum_{i\in C}Q(u_{i})\equiv1\mod2$
for any $C\ne\emptyset$, by making $u\in\{0,3\}$ for all except
one value of $i\in C$, and $u\in\{1,2,4,5\}$ for exactly one such
value. So we have the only nonzero term in this sum is for $C=\emptyset$,
which must be equal to $2^{3^{k}}$ independent of the type.

So then we have 
\begin{eqnarray*}
\sum_{\mathcal{U}}\gamma(\mathcal{U}) & = & 2^{3^{k}}\sum_{s\in S}\sum_{x\in y+\ker M_{k}}\left[\prod_{\{u,v\}}\left(-1\right)^{\sum_{s_{i}=B}\Phi(u_{i},v_{i})x_{u}x_{v}}\right]\\
 &  & \times\left[\sum_{J\subset\{1,...,k\}}\frac{1}{2^{|J|}}\prod_{u}(-1)^{x_{u}\sum_{i=1}^{k}\lambda^{1}(u_{i})\chi_{J}(i)+\lambda^{2}(u_{i})(1-\chi_{J}(u_{i}))}\right]\\
 & = & 2^{3^{k}}\sum_{s\in S}\sum_{x\in y+\ker M_{k}}(-1)^{\sum_{\{u,v\}}\sum_{i:s_{i}=B}\Phi(u_{i},v_{i})x_{u}x_{v}}\\
 &  & \times\left[\sum_{J\subset\{1,...,k\}}\frac{1}{2^{|J|}}(-1)^{\sum_{u}x_{u}\sum_{i=1}^{k}\lambda^{1}(u_{i})\chi_{J}(i)+\lambda^{2}(u_{i})(1-\chi_{J}(u_{i}))}\right]
\end{eqnarray*}
Call $\mathcal{U}_{B}=B^{k}$. Then we can write every $\mathcal{U}=\{(s,u):u\in\mathcal{U}_{B}\}$
where clearly $\Phi((s,u),(s,v))=\sum_{i:s_{i}=B}\Phi(u_{i},v_{i})$.
The same holds for $\lambda^{1},\lambda^{2}$ similarly. We can apply
the binomial theorem to the sum over types $s\in S$ to show 
\begin{eqnarray*}
\sum_{\mathcal{U}}\gamma(\mathcal{U}) & = & 2^{3^{k}}\sum_{x\in y+\ker M_{k}}\sum_{J\subset\{1,...,k\}}\frac{1}{2^{|J|}}\\
 &  & \times\prod_{j=1}^{k}\left(1+(-1)^{\sum_{\{u,v\}\subset\mathcal{U}_{B}}\Phi(u_{j},v_{j})x_{u}x_{v}+\sum_{u\in\mathcal{U}_{B}}x_{u}\lambda^{1}(u_{j})\chi_{J}(j)+x_{u}\lambda^{2}(u_{j})(1-\chi_{J}(j))}\right)
\end{eqnarray*}
Notice we have 
\begin{align*}
\sum_{\{u,v\}\subset\mathcal{U}_{B}}\Phi(u_{1},v_{1})x_{u}x_{v} & =\sum_{u_{1}=0,v_{1}=2}x_{u}x_{v}+\sum_{u_{1}=0,v_{1}=4}x_{u}x_{v}+\sum_{u_{1}=2,v_{1}=4}x_{u}x_{v}\\
 & =\sum_{u_{1}=0}x_{u}\sum_{v_{1}=2}x_{v}+\sum_{u_{1}=0}x_{u}\sum_{v_{1}=4}x_{v}+\sum_{u_{1}=2}x_{u}\sum_{v_{1}=4}x_{v}
\end{align*}
\begin{align*}
\sum_{u\in\mathcal{U}_{B}}x_{u}\lambda^{1}(u_{1})\chi_{J}(1)+x_{u}\lambda^{2}(u_{1})(1-\chi_{J}(1)) & =\begin{cases}
\sum_{u_{1}=2}x_{u}+\sum_{u_{1}=4}x_{u} & 1\in J\\
\sum_{u_{1}=0}x_{u}+\sum_{u_{1}=4}x_{u} & 1\not\in J
\end{cases}
\end{align*}
Notice that we have $\sum_{u_{1}=m}x_{u}\equiv|J|+1\mod2$ if $m=0,2$
and $\equiv0\mod2$ otherwise. It then follows that 
\begin{align*}
\sum_{\{u,v\}\subset\mathcal{U}_{B}}\Phi(u_{1},v_{1})x_{u}x_{v}=(|J|+1)^{2}\equiv|J|+1\mod2\\
\sum_{u\in\mathcal{U}_{B}}x_{u}\lambda^{1}(u_{1})\chi_{J}(1)+x_{u}\lambda^{2}(u_{1})(1-\chi_{J}(1))=|J|+1
\end{align*}
Noting that squaring integers mod 2 does nothing. By symmetry the
same is true for all $j=1,...,k$, so we now have 
\begin{align*}
\sum_{\mathcal{U}}\gamma(\mathcal{U}) & =2^{3^{k}}\sum_{x\in y+\ker M_{k}}\sum_{J\subset\{1,...,k\}}\frac{1}{2^{|J|}}\prod_{j=1}^{k}\left(1+(-1)^{0}\right)\\
 & =2^{3^{k}}\sum_{x\in y+\ker M_{k}}2^{k}\sum_{J\subset\{1,...,k\}}\frac{1}{2^{|J|}}\\
 & =2^{\dim\ker M_{k}+3^{k}+k}\left(\frac{3}{2}\right)^{k}
\end{align*}
Recall that $\dim M_{k}=3^{k}-2k-1$. So we have 
\begin{align*}
\sum_{\mathcal{U}}\gamma(\mathcal{U}) & =2^{2\cdot3^{k}-2k-1}3^{k}
\end{align*}
Thus implying 
\begin{align*}
S_{k,0,0}(X) & =\frac{3^{k}}{2^{5k+1}}\left(\sum_{4n<X}\mu^{2}(2n)(3a)^{k\omega(n)}\right)+O\left(X(\log X)^{(3a)^{k}-a^{k}-1+\epsilon}\right)
\end{align*}

We then have that 
\begin{align*}
\sum_{4n<X}\mu^{2}(2n)(3a)^{k\omega(n)}=\frac{2}{(3a)^{k}}\sum_{d\in\mathcal{D}_{4,8,X}^{-}}(3a)^{k\omega(d)}
\end{align*}
implies the result.

\subsection{The case $d<0$ and $d\equiv0$ mod $8$}

Next we consider fundamental discriminants $d<0$ and $d\equiv4$
mod $8$. The number of $H_{8}$ extensions of a quadratic field $k$
Galois over $\Q$ with such a discriminant is 
\[
f\left(d\right)=f_{1}\left(d\right)-f_{2}\left(d\right)+2^{\omega(d)-4}
\]
where we define

\begin{eqnarray*}
f_{1}\left(d\right) & = & \frac{1}{2}\sum_{d=d_{1}d_{2}d_{3}}\left(1+\Leg{d_{2}d_{3}}{2}\right)\frac{\prod_{i}2^{\omega\left(d_{i}\right)-1}}{2^{\omega\left(d\right)}}\prod_{p\mid d_{3}}\left(1+\left(\frac{-2d_{1}d_{2}}{p}\right)\right)\\
 &  & \times\prod_{p\mid d_{1}}\left(1+\left(\frac{d_{2}d_{3}}{p}\right)\right)\prod_{p\mid d_{2}}\left(1+\left(\frac{-2d_{3}d_{1}}{p}\right)\right)\\
 &  & +\sum_{d=d_{1}d_{2}d_{3}}\left(1+\Leg{-d_{2}d_{3}}{2}\right)\frac{\prod_{i}2^{\omega\left(d_{i}\right)-1}}{2^{\omega\left(d\right)}}\prod_{p\mid d_{3}}\left(1+\left(\frac{-2d_{1}d_{2}}{p}\right)\right)\\
 &  & \times\prod_{p\mid d_{1}}\left(1+\left(\frac{-d_{2}d_{3}}{p}\right)\right)\prod_{p\mid d_{2}}\left(1+\left(\frac{2d_{3}d_{1}}{p}\right)\right)\\
\end{eqnarray*}
and 
\begin{eqnarray*}
f_{2}\left(d\right) & = & \sum_{d=d_{1}d_{2}}\frac{\prod_{i}2^{\omega\left(d_{i}\right)-1}}{2^{\omega\left(d\right)}}\left(1+\Leg{d_{2}}{2}\right)\prod_{p\mid d_{1}}\left(1+\left(\frac{d_{2}}{p}\right)\right)\prod_{p\mid d_{2}}\left(1+\left(\frac{-2d_{1}}{p}\right)\right)\\
 &  & +\sum_{d=d_{1}d_{2}}\frac{\prod_{i}2^{\omega\left(d_{i}\right)-1}}{2^{\omega\left(d\right)}}\left(1+\Leg{-d_{2}}{2}\right)\prod_{p\mid d_{1}}\left(1+\left(\frac{-d_{2}}{p}\right)\right)\prod_{p\mid d_{2}}\left(1+\left(\frac{2d_{1}}{p}\right)\right)\\
\end{eqnarray*}
where the factoriation $d=d_{1}d_{2}d_{3}$ is into integers such
that $d_{1}\equiv-1\mod4$ and $d_{i}\equiv1\mod4$, and each sum
corresponds to $d_{1}<0$ and $d_{i}<0$ for $i\ne1$. Otherwise this
follows from the same reasoning as in the previous subsections. As
before this gives 
\begin{eqnarray*}
f_{1}\left(d\right) & = & \frac{1}{2^{4}}\sum_{m=0}^{1}\sum_{d=-4D_{0}D_{1}D_{2}D_{3}D_{4}D_{5}}\left(1+\Leg{D_{1}D_{2}D_{4}D_{5}}{2}\right)\Leg{2}{D_{2}D_{4}}\\
 &  & \times\left(\frac{1}{2}\Leg{-1}{D_{2}D_{4}}m+\Leg{-1}{D_{0}D_{4}}(1-m)\right)\\
 &  & \times\left(\frac{D_{0}D_{3}D_{2}D_{5}}{D_{4}}\right)\left(\frac{D_{2}D_{5}D_{4}D_{1}}{D_{0}}\right)\left(\frac{D_{4}D_{1}D_{0}D_{3}}{D_{2}}\right)
\end{eqnarray*}
and 
\begin{eqnarray*}
f_{2}\left(d\right) & = & \frac{1}{2^{3}}\sum_{m=0}^{1}\sum_{d=-4E_{0}E_{1}E_{2}E_{3}}\left(1+\Leg{E_{2}E_{3}}{2}\right)\Leg{2}{E_{2}}\\
 &  & \left(\Leg{-1}{E_{2}}m+\Leg{-1}{E_{0}}(1-m)\right)\left(\frac{E_{2}E_{3}}{E_{0}}\right)\left(\frac{E_{0}E_{1}}{E_{2}}\right).
\end{eqnarray*}
where the sum is over factorizations which satisfy the congruences
\begin{eqnarray*}
D_{2}D_{5} & \equiv & (-1)^{m+1}\mod4\\
D_{4}D_{1} & \equiv & 1\mod4
\end{eqnarray*}
and 
\begin{eqnarray*}
E_{2}E_{3} & \equiv & (-1)^{m+1}\mod4
\end{eqnarray*}
As before we compute 
\begin{eqnarray*}
f_{j_{1},j_{2}}\left(d\right) & = & \frac{1}{2^{3j_{1}+3j_{2}}}\sum_{\left(D_{u}\right)}\sum_{C\subset\{1,...,j_{1}+j_{2}\}}\prod_{u}\Leg{D_{u}}{2}^{\sum_{i\in C}Q(u_{i})}\sum_{J_{1}\subset\{1,...,j_{1}\}}\sum_{J_{2}\subset\{j_{1}+1,...,j_{1}+j_{2}\}}\frac{1}{2^{|J_{1}|}}\\
 &  & \times\Leg{-1}{D_{u}}^{\sum_{i=1}^{j_{1}+j_{2}}\lambda^{1}(u_{i})\chi_{J_{1}}(i)+\lambda^{2}(u_{i})(1-\chi_{J_{1}}(u_{i}))+\gamma^{1}(u_{i})\chi_{J_{2}}(i)+\gamma^{2}(u_{i})(1-\chi_{J_{2}}(i))}\\
 &  & \times\Leg{2}{D_{u}}^{\sum_{i=1}^{j_{1}+j_{2}}\lambda^{1}(u_{i})+\gamma^{1}(u_{i})}\prod_{u,v}\left(\frac{D_{u}}{D_{v}}\right)^{\Phi_{j_{1}}\left(u,v\right)+\Psi_{j_{2}}\left(u,v\right)}
\end{eqnarray*}
For $Q(u)=0$ if $u\in\{0,3\}$ and $1$ otherwise, $\lambda^{1}(u)=1$
iff $u=2,4$, $\lambda^{2}(u)=1$ iff $u=0,4$, $\gamma^{1}(u)=1$
iff $u=2$, and $\gamma^{2}(u)=1$ iff $u=0$. The sum is over $6^{j_{1}}4^{j_{2}}$
tuples of integers $\left(D_{u}\right)$ which satisfy $\prod_{u}D_{u}=d$
and the congruence conditions: for all $1\le i\le j_{1}$ and $\left(u_{i},v_{i}\right)\in\left\{ \left(0,3\right),\left(2,5\right),\left(4,1\right)\right\} $
and all $j_{1}+1\le i\le j_{1}+j_{2}$ and $\left(u_{i},v_{i}\right)\in\left\{ \left(0,1\right),\left(2,3\right)\right\} $

\begin{eqnarray}
\prod_{u}D_{u}\prod_{v}D_{v} & \equiv & \begin{cases}
(-1)^{|J_{1}|+1}\mod4 & (u_{i},v_{i})=(0,3),(2,5),i\le j_{1}\\
(-1)^{|J_{2}|+1}\mod4 & (u_{i},v_{i})=(0,1),i\ge j_{2}\\
1\mod4 & \text{else}
\end{cases}
\end{eqnarray}
where the above products are over all $u$ with $u_{i}$ in the $i$th
position and all $v$ with $v_{i}$ in the $i$th position.

Thus multiplying by $2^{kj_{3}\omega(d)}a^{k\omega\left(d\right)}$
and summing over discriminants $d<0$ with $d\equiv1\mod4$ we get
\begin{eqnarray*}
\sum_{d<X}2^{j_{3}(\omega(d))}a^{\omega\left(d\right)}f_{j_{1},j_{2}}\left(d\right) & = & \frac{1}{3^{j_{1}}2^{3j_{1}+3j_{2}}}\sum_{\left(D_{u}\right)}\mu^{2}\left(\prod_{u}D_{u}\right)a^{k\omega\left(d\right)}2^{kj_{3}\omega(d)}\\
 &  & \times\sum_{C\subset\{1,...,j_{1}+j_{2}\}}\sum_{J_{1}\subset\{1,...,j_{1}\}}\sum_{J_{2}\subset\{j_{1}+1,...,j_{1}+j_{2}\}}\prod_{u}\Leg{D_{u}}{2}^{\sum_{i\in C}Q(u_{i})}\\
 &  & \times\frac{1}{2^{|J_{1}|}}\Leg{-1}{D_{u}}^{\sum_{i=1}^{j_{1}+j_{2}}\lambda^{1}(u_{i})\chi_{J_{1}}(i)+\lambda^{2}(u_{i})(1-\chi_{J_{1}}(u_{i}))+\gamma^{1}(u_{i})\chi_{J_{2}}(i)+\gamma^{2}(u_{i})(1-\chi_{J_{2}}(i))}\\
 &  & \times\Leg{2}{D_{u}}^{\sum_{i=1}^{j_{1}+j_{2}}\lambda^{1}(u_{i})+\gamma^{1}(u_{i})}\prod_{u,v}\left(\frac{D_{u}}{D_{v}}\right)^{\Phi_{j_{1}}\left(u,v\right)+\Psi_{j_{2}}\left(u,v\right)}
\end{eqnarray*}
where the sum is over $6^{j_{1}}4^{j_{2}}$ tuples of integers $\left(D_{u}\right)$
which satisfy $\prod_{u}D_{u}<X$ and the conditions $\left(\ref{eq:congruence2}\right)$.

In this case, we have 
\begin{eqnarray*}
\sum_{\mathbf{A}\mathrm{\ admissible\ for\ }\mbox{\ensuremath{\mathcal{U}}}}S_{k,0,0}\left(X,\mathbf{A}\right) & = & \frac{1}{2^{3k}}\left(\sum_{\left(D_{u}\right)}\mu^{2}\left(d/2\right)a^{k\omega\left(d/4\right)}\right)\\
 &  & \times\sum_{\left(h_{u}\right)}\left[\sum_{C\subset\{1,...,k\}}\prod_{u}\left(-1\right)^{\sum_{i}Q(u_{i})\chi_{C}(i)\frac{h_{u}^{2}-1}{8}}\right]\\
 &  & \sum_{J\subset\{1,...,k\}}\frac{1}{2^{|J|}}\prod_{u}(-1)^{\frac{h_{u}-1}{2}\sum_{i=1}^{k}\lambda^{1}(u_{i})\chi_{J}(i)+\lambda^{2}(u_{i})(1-\chi_{J}(u_{i}))}\\
 &  & \times\prod_{u}(-1)^{\sum_{i}\lambda^{1}(u_{i})\frac{h_{u}^{2}-1}{8}}\\
 &  & \times\left[\prod_{u,v}\left(-1\right)^{\Phi_{k}\left(u,v\right)\frac{h_{u}-1}{2}\frac{h_{v}-1}{2}}\right]+O\left(X\left(\log X\right)^{(3a)^{k}-a^{k}-1+\epsilon}\right).
\end{eqnarray*}

For $h_{u}$ the congruence class of $D_{u}\mod8$ and $d=4\prod D_{u}$
and we grouped the $4$ factor with the first discriminant in the
factorization (i.e. for $k=1$ the factorization is $d=(4D_{0}D_{3})(D_{1}D_{4})(D_{2}D_{5})$).

Then, noting that $h_{u}$ is one out of four choices for odd numbers
$\mod8$, we get 
\begin{eqnarray*}
S_{k,0,0}(X) & = & \frac{1}{2^{4k}2^{2\cdot3^{k}}}\left(\sum_{\mathcal{U}}\gamma(\mathcal{U})\right)\left(\sum_{4n<X}\mu^{2}(2n)(3a)^{k\omega(n)}\right)+O\left(X(\log X)^{(3a)^{k}-a^{k}-1+\epsilon}\right)
\end{eqnarray*}
where we define 
\begin{eqnarray*}
\gamma(\mathcal{U}) & = & \sum_{\left(h_{u}\right)}\left[\sum_{C\subset\{1,...,k\}}\prod_{u}\left(-1\right)^{\sum_{i}(Q(u_{i})\chi_{C}(i)+\lambda^{1}(u_{i}))\frac{h_{u}^{2}-1}{8}}\right]\left[\prod_{u,v}\left(-1\right)^{\Phi_{k}\left(u,v\right)\frac{h_{u}-1}{2}\frac{h_{v}-1}{2}}\right]\\
 &  & \times\left[\sum_{J\subset\{1,...,k\}}\frac{1}{2^{|J|}}\prod_{u}(-1)^{\frac{h_{u}-1}{2}\sum_{i=1}^{k}\lambda^{1}(u_{i})\chi_{J}(i)+\lambda^{2}(u_{i})(1-\chi_{J}(u_{i}))}\right]
\end{eqnarray*}
allowing odd congruence classes $h_{u}\mod8$ satisfying the following
conditions: for all $1\le i\le k$ 
\begin{eqnarray}
\prod_{u}h_{u}\prod_{v}h_{v}\equiv\begin{cases}
1\mod4 & (u_{i},v_{i})=(1,4)\\
(-1)^{|J|+1}\mod4 & (u_{i},v_{i})=(2,5)\\
(-1)^{|J|+1}\prod_{u'}h_{u'}\prod_{v'}h_{v'} & (u_{i},v_{i})=(u'_{j},v'_{j})=(0,3)\\
 & \text{ for some }j\ne i
\end{cases}
\end{eqnarray}
where the above products are taken over all $u$ with $u_{i}$ in
the $i$th position and $v$ with $v_{i}$ in the $i$th position.
(Recalling that the odd parts of even discriminants are $-1\mod4$
if $8\nmid d$). If we call $x\equiv\left(\frac{h_{u}-1}{2}\right)\mod2$
an element of $\F_{2}^{3^{k}}$, then $x\in y+\ker M_{k}$ some coset
of $\ker M_{k}$. There are two such cosets depending on the congruence
class of $\prod_{u}h_{u}\prod_{v}h_{v}$ for $(u_{i},v_{i})=(0,3)$.

Now consider that $\Phi(u,v)=0$ if $u,v\in A=\{1,3,5\}$. $\mathcal{U}$
is a largest maximal unlinked set, and so has a type $s\in S$, so
it follows that 
\[
\Phi_{k}(u,v)=\sum_{i}\Phi(u_{i},v_{i})=\sum_{i:s_{i}=B}\Phi(u_{i},v_{i})
\]
This way we show that 
\begin{eqnarray*}
\sum_{\mathcal{U}}\gamma(\mathcal{U}) & = & \sum_{s\in S}\sum_{y}\sum_{x\in y+\ker M_{k}}\left[\sum_{(h_{u}):\frac{h_{u}-1}{2}\equiv x_{u}\mod2}\sum_{C\subset\{1,...,k\}}\prod_{u}\left(-1\right)^{\sum_{i\in C}Q(u_{i})\left(\frac{h_{u}^{2}-1}{8}\right)}\right]\\
 &  & \times\left[\prod_{u}(-1)^{\sum_{i}\lambda^{1}(u_{i})\frac{h_{u}^{2}-1}{8}}\right]\left[\prod_{\{u,v\}}\left(-1\right)^{\sum_{s_{i}=B}\Phi(u_{i},v_{i})x_{u}x_{v}}\right]\\
 &  & \times\left[\sum_{J\subset\{1,...,k\}}\frac{1}{2^{|J|}}\prod_{u}(-1)^{x_{u}\sum_{i=1}^{k}\lambda^{1}(u_{i})\chi_{J}(i)+\lambda^{2}(u_{i})(1-\chi_{J}(u_{i}))}\right]
\end{eqnarray*}
Notice that for each $x_{u}$, there are two choices of $h_{u}\mod8$
such that $\frac{h_{u}-1}{2}\equiv x_{u}\mod2$, because $h_{u}$
and $5h_{u}$ give the same image. If we fix an $(h(x)_{u})$ satisfying
this property without loss of generality also satisfying $\frac{h_{u}^{2}-1}{8}\equiv0\mod2$,
then we have 
\begin{align*}
\sum_{(h_{u}):\frac{h_{u}-1}{2}\equiv x_{u}\mod2} & \sum_{C\subset\{1,...,k\}}\prod_{u}\left(-1\right)^{\sum_{i}(Q(u_{i})\chi_{C}(i)+\lambda^{1}(u_{i}))\left(\frac{h_{u}^{2}-1}{8}\right)}\\
 & =\sum_{T\subset\mathcal{U}}\sum_{C\subset\{1,...,k\}}\prod_{u}\left(-1\right)^{\sum_{i}(Q(u_{i})\chi_{C}(i)+\lambda^{1}(u_{i}))\left(\frac{(h_{u}5^{\chi_{T}(u)})^{2}-1}{8}\right)}\\
 & =\sum_{T\subset\mathcal{U}}\sum_{C\subset\{1,...,k\}}\prod_{u}\left(-1\right)^{\sum_{i}(Q(u_{i})\chi_{C}(i)+\lambda^{1}(u_{i}))\left(\frac{5^{2\chi_{T}(u)}-1}{8}\right)}
\end{align*}
Now, $\frac{5^{2}-1}{8}\equiv1\mod2$. So it follows that $\left(\frac{5^{2\chi_{T}(u)}-1}{8}\right)\equiv(\chi_{T}(u))\mod2$
so we have 
\begin{align*}
\sum_{(h_{u}):\frac{h_{u}-1}{2}\equiv x_{u}\mod2} & \sum_{C\subset\{1,...,k\}}\prod_{u}\left(-1\right)^{\sum_{i}(Q(u_{i})\chi_{C}(i)+\lambda^{1}(u_{i}))\left(\frac{h_{u}^{2}-1}{8}\right)}\\
 & =\sum_{T\subset\mathcal{U}}\sum_{C\subset\{1,...,k\}}\prod_{u}\left(-1\right)^{\sum_{i}(Q(u_{i})\chi_{C}(i)+\lambda^{1}(u_{i}))\chi_{T}(u)}\\
 & =\sum_{T\subset\mathcal{U}}\sum_{C\subset\{1,...,k\}}\left(-1\right)^{\sum_{u,i}(Q(u_{i})\chi_{C}(i)+\lambda^{1}(u_{i}))\chi_{T}(u)}\\
 & =\sum_{C\subset\{1,...,k\}}\prod_{u}\left(1+(-1)^{\sum_{i}Q(u_{i})\chi_{C}(i)+\lambda^{1}(u_{i})}\right)
\end{align*}
By using the binomial theorem in two different ways.

We then have 
\begin{align*}
Q(u_{i})\chi_{C}(i)+\lambda(u_{i}) & =\begin{cases}
0 & u_{i}\in\{0,3,2,4\},i\in C\\
1 & u_{i}\in\{1,5\},i\in C\\
0 & u_{i}\in\{0,3,1,5\},i\not\in C\\
1 & u_{i}\in\{2,4\},i\not\in C
\end{cases}
\end{align*}
Suppose that $C\not\subset\{i:s_{i}=B\}$, then fix some $j\in C$
such that $s_{j}=A$. Choose $u\in\mathcal{U}$ such that 
\begin{align*}
u_{i}\in\begin{cases}
\{1\} & i=j\\
\{0,3\} & i\ne j
\end{cases}
\end{align*}
Then we get the summands are zero corresponding to these $C$. Suppose
next that $\{i:s_{i}=B\}\not\subset C$, then fix $j\not\in C$ such
that $s_{j}=B$ and choose $u\in\mathcal{U}$ such that 
\begin{align*}
u_{i}\in\begin{cases}
\{2\} & i=j\\
\{0,3\} & i\ne j
\end{cases}
\end{align*}
Then these summands give us zero as well. The only summand remaining
is $C=\{i:s_{i}=B\}$, which clearly gives $2^{3^{k}}$.

So then we have 
\begin{eqnarray*}
\sum_{\mathcal{U}}\gamma(\mathcal{U}) & = & 2^{3^{k}}\sum_{s\in S}\sum_{y}\sum_{x\in y+\ker M_{k}}\left[\prod_{\{u,v\}}\left(-1\right)^{\sum_{s_{i}=B}\Phi(u_{i},v_{i})x_{u}x_{v}}\right]\\
 &  & \times\left[\sum_{J\subset\{1,...,k\}}\frac{1}{2^{|J|}}\prod_{u}(-1)^{x_{u}\sum_{i=1}^{k}\lambda^{1}(u_{i})\chi_{J}(i)+\lambda^{2}(u_{i})(1-\chi_{J}(u_{i}))}\right]\\
 & = & 2^{3^{k}}\sum_{s\in S}\sum_{y}\sum_{x\in y+\ker M_{k}}(-1)^{\sum_{\{u,v\}}\sum_{i:s_{i}=B}\Phi(u_{i},v_{i})x_{u}x_{v}}\\
 &  & \times\left[\sum_{J\subset\{1,...,k\}}\frac{1}{2^{|J|}}(-1)^{\sum_{u}x_{u}\sum_{i=1}^{k}\lambda^{1}(u_{i})\chi_{J}(i)+\lambda^{2}(u_{i})(1-\chi_{J}(u_{i}))}\right]
\end{eqnarray*}
Call $\mathcal{U}_{B}=B^{k}$. Then we can write every $\mathcal{U}=\{(s,u):u\in\mathcal{U}_{B}\}$
where clearly $\Phi((s,u),(s,v))=\sum_{i:s_{i}=B}\Phi(u_{i},v_{i})$.
The same holds for $\lambda^{1},\lambda^{2}$ similarly. We can apply
the binomial theorem to the sum over types $s\in S$ to show 
\begin{eqnarray*}
\sum_{\mathcal{U}}\gamma(\mathcal{U}) & = & 2^{3^{k}}\sum_{y}\sum_{x\in y+\ker M_{k}}\sum_{J\subset\{1,...,k\}}\frac{1}{2^{|J|}}\\
 &  & \times\prod_{j=1}^{k}\left(1+(-1)^{\sum_{\{u,v\}\subset\mathcal{U}_{B}}\Phi(u_{j},v_{j})x_{u}x_{v}+\sum_{u\in\mathcal{U}_{B}}x_{u}\lambda^{1}(u_{j})\chi_{J}(j)+x_{u}\lambda^{2}(u_{j})(1-\chi_{J}(j))}\right)
\end{eqnarray*}
Notice we have 
\begin{align*}
\sum_{\{u,v\}\subset\mathcal{U}_{B}}\Phi(u_{1},v_{1})x_{u}x_{v} & =\sum_{u_{1}=0,v_{1}=2}x_{u}x_{v}+\sum_{u_{1}=0,v_{1}=4}x_{u}x_{v}+\sum_{u_{1}=2,v_{1}=4}x_{u}x_{v}\\
 & =\sum_{u_{1}=0}x_{u}\sum_{v_{1}=2}x_{v}+\sum_{u_{1}=0}x_{u}\sum_{v_{1}=4}x_{v}+\sum_{u_{1}=2}x_{u}\sum_{v_{1}=4}x_{v}
\end{align*}
\begin{align*}
\sum_{u\in\mathcal{U}_{B}}x_{u}\lambda^{1}(u_{1})\chi_{J}(1)+x_{u}\lambda^{2}(u_{1})(1-\chi_{J}(1)) & =\begin{cases}
\sum_{u_{1}=2}x_{u}+\sum_{u_{1}=4}x_{u} & 1\in J\\
\sum_{u_{1}=0}x_{u}+\sum_{u_{1}=4}x_{u} & 1\not\in J
\end{cases}
\end{align*}
Notice that we have $\sum_{u_{1}=m}x_{u}\equiv|J|+1\mod2$ if $m=0,2$
and $\equiv0\mod2$ otherwise. It then follows that 
\begin{align*}
\sum_{\{u,v\}\subset\mathcal{U}_{B}}\Phi(u_{1},v_{1})x_{u}x_{v}=(|J|+1)^{2}\equiv|J|+1\mod2\\
\sum_{u\in\mathcal{U}_{B}}x_{u}\lambda^{1}(u_{1})\chi_{J}(1)+x_{u}\lambda^{2}(u_{1})(1-\chi_{J}(1))=|J|+1
\end{align*}
Noting that squaring integers mod 2 does nothing. By symmetry the
same is true for all $j=1,...,k$, so we now have 
\begin{align*}
\sum_{\mathcal{U}}\gamma(\mathcal{U}) & =2^{3^{k}}\sum_{y}\sum_{x\in y+\ker M_{k}}\sum_{J\subset\{1,...,k\}}\frac{1}{2^{|J|}}\prod_{j=1}^{k}\left(1+(-1)^{0}\right)\\
 & =2^{3^{k}}\sum_{y}\sum_{x\in y+\ker M_{k}}2^{k}\sum_{J\subset\{1,...,k\}}\frac{1}{2^{|J|}}\\
 & =2^{\dim\ker M_{k}+3^{k}+k+1}\left(\frac{3}{2}\right)^{k}
\end{align*}
Noting that there are two choices for $y$. Recall that $\dim M_{k}=3^{k}-2k-1$.
So we have 
\begin{align*}
\sum_{\mathcal{U}}\gamma(\mathcal{U}) & =2^{2\cdot3^{k}-2k}3^{k}
\end{align*}
Thus implying 
\begin{align*}
S_{k,0,0}(X) & =\frac{3^{k}}{2^{5k}}\left(\sum_{8n<X}\mu^{2}(2n)(3a)^{k\omega(n)}\right)+O\left(X(\log X)^{(3a)^{k}-a^{k}-1+\epsilon}\right)
\end{align*}

We then have that 
\begin{align*}
\sum_{8n<X}\mu^{2}(2n)(3a)^{k\omega(n)}=\frac{1}{(3a)^{k}}\sum_{d\in\mathcal{D}_{0,8,X}^{-}}(3a)^{k\omega(d)}
\end{align*}
implies the result.

\subsection{The case $d>0$ and $d\equiv1$ mod $4$}

Next we consider fundamental discriminants $d>0$ and $d\equiv1$
mod $4$. The number of $H_{8}$ extensions of a quadratic field $k$
Galois over $\Q$ with such a discriminant is 
\[
f\left(d\right)=f_{1}\left(d\right)-f_{2}\left(d\right)+2^{\omega(d)-4}
\]
where we define

\begin{eqnarray*}
f_{1}\left(d\right) & = & \frac{1}{6}2^{-4}\sum_{d=d_{1}d_{2}d_{3}}\prod_{p\mid d_{3}}\left(1+\left(\frac{d_{1}d_{2}}{p}\right)\right)\prod_{p\mid d_{1}}\left(1+\left(\frac{d_{2}d_{3}}{p}\right)\right)\prod_{p\mid d_{2}}\left(1+\left(\frac{d_{3}d_{1}}{p}\right)\right)
\end{eqnarray*}
and 
\begin{eqnarray*}
f_{2}\left(d\right) & = & \frac{1}{2}2^{-4}\sum_{d=d_{1}d_{2}}\prod_{p\mid d_{1}}\left(1+\left(\frac{d_{2}}{p}\right)\right)\prod_{p\mid d_{2}}\left(1+\left(\frac{d_{1}}{p}\right)\right)
\end{eqnarray*}
where the factoriation $d=d_{1}d_{2}d_{3}$ is into integers such
that each $d_{i}\equiv1\mod4$ by the same reasoning as in the previous
subsection, and $\alpha(d)=0$ if all primes $p\mid d$ satisfy $p\equiv1\mod4$
and $\alpha(d)=1$ otherwise. As before this gives 
\begin{eqnarray*}
f_{1}\left(d\right)=\frac{1}{3\cdot2^{5}}\sum_{d=D_{0}D_{1}D_{2}D_{3}D_{4}D_{5}}\left(\frac{D_{0}D_{3}D_{2}D_{5}}{D_{4}}\right)\left(\frac{D_{2}D_{5}D_{4}D_{1}}{D_{0}}\right)\left(\frac{D_{4}D_{1}D_{0}D_{3}}{D_{2}}\right)
\end{eqnarray*}
and 
\begin{eqnarray*}
f_{2}\left(d\right)=\frac{1}{2^{4}}\sum_{d=E_{0}E_{1}E_{2}E_{3}}\left(\frac{E_{2}E_{3}}{E_{0}}\right)\left(\frac{E_{0}E_{1}}{E_{2}}\right).
\end{eqnarray*}
where the sum is over factorizations which satisfy the congruences
\begin{eqnarray*}
D_{0}D_{3},D_{2}D_{5},D_{1}D_{1} & \equiv & 1\mod4
\end{eqnarray*}
\begin{eqnarray*}
E_{0}E_{1},E_{2}E_{3} & \equiv & 1\mod4.
\end{eqnarray*}
As before we compute 
\[
f_{j_{1},j_{2}}\left(d\right)=\frac{1}{3^{j_{1}}2^{5j_{1}+4j_{2}}}\sum_{\left(D_{u}\right)}\prod_{u,v}\left(\frac{D_{u}}{D_{v}}\right)^{\Phi_{j_{1}}\left(u,v\right)+\Psi_{j_{2}}\left(u,v\right)}
\]
where now the sum is over $6^{j_{1}}4^{j_{2}}$ tuples of integers
$\left(D_{u}\right)$ which satisfy $\prod_{u}D_{u}=d$ and the congruence
conditions: for all $1\le i\le j_{1}$ and $\left(u_{i},v_{i}\right)\in\left\{ \left(0,3\right),\left(2,5\right),\left(4,1\right)\right\} $
and all $j_{1}+1\le i\le j_{1}+j_{2}$ and $\left(u_{i},v_{i}\right)\in\left\{ \left(0,1\right),\left(2,3\right)\right\} $

\begin{eqnarray}
\prod_{u}D_{u}\prod_{v}D_{v} & \equiv & 1\mod4\label{eq:positive1mod4congruence}
\end{eqnarray}
where the above products are over all $u$ with $u_{i}$ in the $i$th
position and all $v$ with $v_{i}$ in the $i$th position.

Thus multiplying by $2^{j_{3}\omega(d)-j_{3}}a^{k\omega\left(d\right)}$
and summing over discriminants $d<0$ with $d\equiv1\mod4$ we get
\begin{eqnarray}
\sum_{d<X}2^{j_{3}\omega(d)-j_{3}\alpha(d)}a^{\omega\left(d\right)}f_{j_{1},j_{2}}\left(d\right)=\label{eq:mainexpression2}\\
\frac{1}{3^{j_{1}}2^{4j_{1}+3j_{2}+k}}\sum_{\left(D_{u}\right)}\mu^{2}\left(\prod_{u}D_{u}\right)a^{k\omega\left(d\right)} & 2^{j_{3}\omega(d)}\prod_{u,v}\left(\frac{D_{u}}{D_{v}}\right)^{\Phi_{j_{1}}\left(u,v\right)+\Psi_{j_{2}}\left(u,v\right)}\nonumber 
\end{eqnarray}
where the sum is over $6^{j_{1}}4^{j_{2}}$ tuples of integers $\left(D_{u}\right)$
which satisfy $\prod_{u}D_{u}<X$ and the conditions $\left(\ref{eq:positive1mod4congruence}\right)$.

Here $S_{k,0,0}\left(X,\mathbf{A}\right)$ reduces to 
\begin{eqnarray*}
S_{k,0,0}\left(X,\mathbf{A}\right) & = & \sum_{\left(h_{u}\right)}\frac{1}{3^{k}2^{4k}}\sum_{\left(D_{u}\right)}2^{-k}\mu^{2}\left(d\right)a^{k\omega\left(d\right)}\left[\prod_{u,v}\left(-1\right)^{\Phi_{k}\left(u,v\right)\frac{h_{u}-1}{2}\frac{h_{v}-1}{2}}\right]\\
 &  & +O\left(X\left(\log X\right)^{(3a)^{k}-a^{k}-1+\epsilon}\right)
\end{eqnarray*}
and by the same procedure as in Section \ref{subsec:negative1mod4moments}
we obtain 
\[
S_{k,0,0}\left(X\right)=\frac{1}{3^{k}2^{5k+3^{k}}}\left(\sum_{\mathcal{U}}\gamma\left(\mathcal{U}\right)\right)\left(\sum_{n<X}\mu^{2}(2n)(3a)^{k\omega(n)}\right)+O\left(X\left(\log X\right)^{(3a)^{k}-a^{k}-1+\epsilon}\right)
\]
where now define 
\[
\gamma\left(\mathcal{U}\right)=\sum_{\left(h_{u}\right)}\left[\prod_{u,v}\left(-1\right)^{\Phi_{k}\left(u,v\right)\frac{h_{u}-1}{2}\frac{h_{v}-1}{2}}\right].
\]
Recall (see $\left(\ref{eq:positive1mod4congruence}\right)$) that
we are allowing all possible congruence classes $\left(h_{u}\right)$
satisfying the conditions: for all $1\le i\le k$ and $\left(u_{i},v_{i}\right)\in\left\{ \left(0,3\right),\left(2,5\right),\left(4,1\right)\right\} $

\begin{equation}
\prod_{u}h_{u}\prod_{v}h_{v}\equiv1\mod4\label{eq:positive1mod4congruence_k}
\end{equation}
where the above products are over all $u$ with $u_{i}$ in the $i$th
position and all $v$ with $v_{i}$ in the $i$th position.

As in Section \ref{subsec:negative1mod4moments} we identify tuples
$\left(h_{u}\right)$ with elements $x\in\mathbb{F}_{2}^{3^{k}}$.
In this case the $x$ satisfying condition $\left(\ref{eq:positive1mod4congruence_k}\right)$
are exactly the kernel of the matrix $M_{k}$ defined in the previous
section. The following proposition is proved almost identically to
Proposition \ref{prop:gamma1mod4negative}. 
\begin{prop}
For all $k\ge1$ 
\[
\sum_{\mathcal{U}}\gamma\left(\mathcal{U}\right)=2^{3^{k}-k-1}.
\]
\end{prop}
Thus we have proven:

\begin{equation}
S_{k,0,0}\left(X\right)=\frac{1}{3^{k}2^{6k+1}}\left(\sum_{n<X}\mu^{2}(2n)(3a)^{k\omega(n)}\right)+O\left(X\left(\log X\right)^{(3a)^{k}-a^{k}-1+\epsilon}\right).
\end{equation}
from which the corresponding case of Theorem \ref{thm:mainthm3} follows
by noting that 
\begin{align*}
\sum_{n<X}\mu^{2}(2n)(3a)^{k\omega(n)} & =2\sum_{d\in\mathcal{D}_{X,1,4}^{+}}(3a)^{k\omega(d)}+o(X)
\end{align*}

\subsection{\label{subsec:positive4mod8}The case $d>0$ and $d\equiv4$ mod
$8$}

Next we consider fundamental discriminants $d>0$ and $d\equiv4$
mod $8$. The number of $H_{8}$ extensions of a quadratic field $k$
Galois over $\Q$ with such a discriminant is 
\[
f\left(d\right)=f_{1}\left(d\right)-f_{2}\left(d\right)+2^{\omega(d)-4}
\]
where we define

\begin{eqnarray*}
f_{1}\left(d\right) & = & \frac{1}{2}2^{-4}\sum_{d=4d_{1}d_{2}d_{3}}\left(1+\Leg{d_{2}d_{3}}{2}\right)\prod_{p\mid d_{3}}\left(1+\left(\frac{d_{1}d_{2}}{p}\right)\right)\\
 &  & \times\prod_{p\mid d_{1}}\left(1+\left(\frac{d_{2}d_{3}}{p}\right)\right)\prod_{p\mid d_{2}}\left(1+\left(\frac{d_{3}d_{1}}{p}\right)\right)
\end{eqnarray*}
and 
\begin{eqnarray*}
f_{2}\left(d\right) & = & \frac{1}{2}2^{-4}\sum_{d=4d_{1}d_{2}}\left(1+\Leg{d_{2}}{2}\right)\prod_{p\mid d_{1}}\left(1+\left(\frac{d_{2}}{p}\right)\right)\prod_{p\mid d_{2}}\left(1+\left(\frac{d_{1}}{p}\right)\right)
\end{eqnarray*}
where the factorization $d=d_{1}d_{2}d_{3}$ is into integers such
that $d_{1}\equiv-1\mod4$ and $d_{i}\equiv1\mod4$ otherwise by the
same reasoning as in the previous subsections, noting that $\alpha(d)=1$.
As before this gives 
\begin{eqnarray*}
f_{1}\left(d\right)=\frac{1}{2^{5}}\sum_{d=D_{0}D_{1}D_{2}D_{3}D_{4}D_{5}}\left(1+\Leg{D_{1}D_{2}D_{4}D_{5}}{2}\right)\left(\frac{D_{0}D_{3}D_{2}D_{5}}{D_{4}}\right)\left(\frac{D_{2}D_{5}D_{4}D_{1}}{D_{0}}\right)\left(\frac{D_{4}D_{1}D_{0}D_{3}}{D_{2}}\right)
\end{eqnarray*}
and 
\begin{eqnarray*}
f_{2}\left(d\right)=\frac{1}{2^{4}}\sum_{d=E_{0}E_{1}E_{2}E_{3}}\left(1+\Leg{E_{2}E_{3}}{2}\right)\left(\frac{E_{2}E_{3}}{E_{0}}\right)\left(\frac{E_{0}E_{1}}{E_{2}}\right).
\end{eqnarray*}
where the sum is over factorizations which satisfy the congruences
\begin{eqnarray*}
D_{0}D_{3} & \equiv & -1\mod4\\
D_{2}D_{5},D_{1}D_{4} & \equiv & 1\mod4
\end{eqnarray*}
\begin{eqnarray*}
E_{0}E_{1} & \equiv & -1\mod4\\
E_{2}E_{3} & \equiv & 1\mod4.
\end{eqnarray*}
Now define where 
\[
Q(u)=\begin{cases}
1 & u\in\left\{ 1,2,4,5\right\} \\
0 & u\in\{0,3\}
\end{cases}
\]
and for any $B\subset\{1,...,j_{1}\}$ let $Q_{B}\left(u\right)=\sum_{i\in B}Q\left(u_{i}\right)$.
There is a similar definition for a function $S_{C}\left(u\right)$
for any $C\subset\{j_{1}+1,...,j_{2}\}$ which appears in the following
expression but the exact form is not important as we will eventually
only be concerned with the case $j_{2}=0$. As before we compute 
\begin{align*}
f_{j_{1},j_{2}}\left(d\right) & =\frac{1}{2^{5j_{1}+4j_{2}}}\sum_{\left(D_{u}\right)}\sum_{B\subset\{1,...,j_{1}\}}\sum_{C\subset\{j_{1}+1,...,j_{2}\}}\prod_{u}\Leg{D_{u}}{2}^{Q_{B}\left(u\right)+S_{C}\left(u\right)}\\
 & \times\prod_{u,v}\left(\frac{D_{u}}{D_{v}}\right)^{\Phi_{j_{1}}\left(u,v\right)+\Psi_{j_{2}}\left(u,v\right)}
\end{align*}
where the sum is over $6^{j_{1}}4^{j_{2}}$ tuples of integers $\left(D_{u}\right)$
which satisfy $\prod_{u}D_{u}=d$ and the congruence conditions: for
all $1\le i\le j_{1}$ and $\left(u_{i},v_{i}\right)\in\left\{ \left(0,3\right),\left(2,5\right),\left(4,1\right)\right\} $

\begin{eqnarray}
\prod_{u}D_{u}\prod_{v}D_{v} & \equiv & \begin{cases}
-1\mod4 & \text{if }(u_{i},v_{i})=(0,3)\\
1\mod4 & \mbox{if }\left(u_{i},v_{i}\right)=\left(2,5\right),\left(4,1\right)
\end{cases}\label{eq:positive4mod8congruence1}
\end{eqnarray}
and all $j_{1}+1\le i\le j_{1}+j_{2}$ and $\left(u_{i},v_{i}\right)\in\left\{ \left(0,1\right),\left(2,3\right)\right\} $

\begin{eqnarray}
\prod_{u}D_{u}\prod_{v}D_{v} & \equiv & \equiv\begin{cases}
-1\mod4 & \text{if }(u_{i},v_{i})=(0,1)\\
1\mod4 & \mbox{if }\left(u_{i},v_{i}\right)=\left(2,3\right)
\end{cases}\label{eq:positive4mod8congruence2}
\end{eqnarray}
and the above products are over all $u$ with $u_{i}$ in the $i$th
position and all $v$ with $v_{i}$ in the $i$th position.

Thus summing over discriminants $d<0$ with $d\equiv1\mod4$ we get
\begin{eqnarray*}
\sum_{d<X}2^{j_{3}\omega(d)}a^{\omega\left(d\right)}f_{j_{1},j_{2}}\left(d\right) & = & \frac{1}{3^{j_{1}}2^{5j_{1}+4j_{2}}}\sum_{\left(D_{u}\right)}\mu^{2}\left(\prod_{u}D_{u}\right)a^{k\omega\left(d\right)}2^{j_{3}\omega(d)}\\
 &  & \times\sum_{B\subset\{1,...,j_{1}\}}\sum_{C\subset\{j_{1}+1,...,j_{2}\}}\prod_{u}\Leg{D_{u}}{2}^{Q_{B}\left(u\right)+S_{C}\left(u\right)}\\
 &  & \times\prod_{u,v}\left(\frac{D_{u}}{D_{v}}\right)^{\Phi_{j_{1}}\left(u,v\right)+\Psi_{j_{2}}\left(u,v\right)}
\end{eqnarray*}
where the sum is over $6^{j_{1}}4^{j_{2}}$ tuples of integers $\left(D_{u}\right)$
which satisfy $\prod_{u}D_{u}<X$ and the conditions $\left(\ref{eq:positive4mod8congruence1}\right)$.

Here $S_{k,0,0}\left(X,\mathbf{A}\right)$ reduces to 
\begin{eqnarray*}
\sum_{\mathbf{A}\mathrm{\ admissible\ for\ }\mbox{\ensuremath{\mathcal{U}}}}S_{k,0,0}\left(X,\mathbf{A}\right) & = & \frac{1}{2^{5k}}\left(\sum_{\left(D_{u}\right)}\mu^{2}\left(d/2\right)a^{k\omega\left(d/4\right)}\right)\\
 &  & \times\sum_{\left(h_{u}\right)}\left[\sum_{C\subset\{1,...,k\}}\prod_{u}\left(-1\right)^{\sum_{i}Q(u_{i})\chi_{C}(i)\frac{h_{u}^{2}-1}{8}}\right]\\
 &  & \times\left[\prod_{u,v}\left(-1\right)^{\Phi_{k}\left(u,v\right)\frac{h_{u}-1}{2}\frac{h_{v}-1}{2}}\right]+O\left(X\left(\log X\right)^{(3a)^{k}-a^{k}-1+\epsilon}\right).
\end{eqnarray*}

For $h_{u}$ the congruence class of $D_{u}\mod8$ and $d=4\prod D_{u}$
and we grouped the $4$ factor with the first discriminant in the
factorization (i.e. for $k=1$ the factorization is $d=(4D_{0}D_{3})(D_{1}D_{4})(D_{2}D_{5})$).

Then, noting that $h_{u}$ is one out of four choices for odd numbers
$\mod8$, we get 
\begin{eqnarray*}
S_{k,0,0}(X) & = & \frac{1}{2^{5k}2^{2\cdot3^{k}}}\left(\sum_{\mathcal{U}}\gamma(\mathcal{U})\right)\left(\sum_{4n<X}\mu^{2}(2n)(3a)^{k\omega(n)}\right)+O\left(X(\log X)^{(3a)^{k}-a^{k}-1+\epsilon}\right)
\end{eqnarray*}
where we define 
\begin{eqnarray*}
\gamma(\mathcal{U}) & = & \sum_{\left(h_{u}\right)}\left[\sum_{C\subset\{1,...,k\}}\prod_{u}\left(-1\right)^{\sum_{i}Q(u_{i})\chi_{C}(i)\frac{h_{u}^{2}-1}{8}}\right]\left[\prod_{u,v}\left(-1\right)^{\Phi_{k}\left(u,v\right)\frac{h_{u}-1}{2}\frac{h_{v}-1}{2}}\right]
\end{eqnarray*}
allowing odd congruence classes $h_{u}\mod8$ satisfying the following
conditions: for all $1\le i\le k$ 
\begin{eqnarray}
\prod_{u}h_{u}\prod_{v}h_{v}\equiv\begin{cases}
1\mod4 & (u_{i},v_{i})\in\{(1,4),(2,5)\}\\
-1\mod4 & (u_{i},v_{i})\in\{(0,3)\}
\end{cases}
\end{eqnarray}
where the above products are taken over all $u$ with $u_{i}$ in
the $i$th position and $v$ with $v_{i}$ in the $i$th position.
(Recalling that the odd parts of even discriminants are $-1\mod4$
if $8\nmid d$). If we call $x\equiv\left(\frac{h_{u}-1}{2}\right)\mod2$
an element of $\F_{2}^{3^{k}}$, then $x\in y+\ker M_{k}$ a coset
of $\ker M_{k}$.

Now consider that $\Phi(u,v)=0$ if $u,v\in A=\{1,3,5\}$. $\mathcal{U}$
is a largest maximal unlinked set, and so has a type $s\in S$, so
it follows that 
\[
\Phi_{k}(u,v)=\sum_{i}\Phi(u_{i},v_{i})=\sum_{i:s_{i}=B}\Phi(u_{i},v_{i})
\]
This way we show that 
\begin{eqnarray*}
\sum_{\mathcal{U}}\gamma(\mathcal{U}) & = & \sum_{s\in S}\sum_{x\in y+\ker M_{k}}\left[\sum_{(h_{u}):\frac{h_{u}-1}{2}\equiv x_{u}\mod2}\sum_{C\subset\{1,...,k\}}\prod_{u}\left(-1\right)^{\sum_{i\in C}Q(u_{i})\left(\frac{h_{u}^{2}-1}{8}\right)}\right]\\
 &  & \times\left[\prod_{\{u,v\}}\left(-1\right)^{\sum_{s_{i}=B}\Phi(u_{i},v_{i})x_{u}x_{v}}\right]
\end{eqnarray*}
Notice that for each $x_{u}$, there are two choices of $h_{u}\mod8$
such that $\frac{h_{u}-1}{2}\equiv x_{u}\mod2$, because $h_{u}$
and $5h_{u}$ give the same image. If we fix an $(h(x)_{u})$ satisfying
this property without loss of generality also satisfying $\frac{h_{u}^{2}-1}{8}\equiv0\mod2$,
then we have 
\begin{align*}
\sum_{(h_{u}):\frac{h_{u}-1}{2}\equiv x_{u}\mod2}\sum_{C\subset\{1,...,k\}}\prod_{u}\left(-1\right)^{\sum_{i\in C}Q(u_{i})\left(\frac{h_{u}^{2}-1}{8}\right)} & =\sum_{T\subset\mathcal{U}}\sum_{C\subset\{1,...,k\}}\prod_{u}\left(-1\right)^{\sum_{i\in C}Q(u_{i})\left(\frac{(h_{u}5^{\chi_{T}(u)})^{2}-1}{8}\right)}\\
 & =\sum_{T\subset\mathcal{U}}\sum_{C\subset\{1,...,k\}}\prod_{u}\left(-1\right)^{\sum_{i\in C}Q(u_{i})\left(\frac{5^{2\chi_{T}(u)}-1}{8}\right)}
\end{align*}
Now, $\frac{5^{2}-1}{8}\equiv1\mod2$. So it follows that $\left(\frac{5^{2\chi_{T}(u)}-1}{8}\right)\equiv(\chi_{T}(u))\mod2$
so we have 
\begin{align*}
\sum_{(h_{u}):\frac{h_{u}-1}{2}\equiv x_{u}\mod2}\sum_{C\subset\{1,...,k\}}\prod_{u}\left(-1\right)^{\sum_{i\in C}Q(u_{i})\left(\frac{h_{u}^{2}-1}{8}\right)} & =\sum_{T\subset\mathcal{U}}\sum_{C\subset\{1,...,k\}}\prod_{u}\left(-1\right)^{\sum_{i\in C}Q(u_{i})\chi_{T}(u)}\\
 & =\sum_{T\subset\mathcal{U}}\sum_{C\subset\{1,...,k\}}\left(-1\right)^{\sum_{u,i}Q(u_{i})\chi_{C}(i)\chi_{T}(u)}\\
 & =\sum_{C\subset\{1,...,k\}}\prod_{u}\left(1+(-1)^{\sum_{i}Q(u_{i})\chi_{C}(i)}\right)
\end{align*}
By using the binomial theorem in two different ways. Now, for $s_{i}$
being $A$ or $B$, we can find a $u\in\mathcal{U}$ such that $\sum_{i\in C}Q(u_{i})\equiv1\mod2$
for any $C\ne\emptyset$, by making $u\in\{0,3\}$ for all except
one value of $i\in C$, and $u\in\{1,2,4,5\}$ for exactly one such
value. So we have the only nonzero term in this sum is for $C=\emptyset$,
which must be equal to $2^{3^{k}}$ independent of the type.

So then we have 
\begin{align*}
\sum_{\mathcal{U}}\gamma(\mathcal{U}) & =2^{3^{k}}\sum_{s\in S}\sum_{x\in y+\ker M_{k}}\left[\prod_{\{u,v\}}\left(-1\right)^{\sum_{s_{i}=B}\Phi(u_{i},v_{i})x_{u}x_{v}}\right]\\
 & =2^{3^{k}}\sum_{s\in S}\sum_{x\in y+\ker M_{k}}(-1)^{\sum_{\{u,v\}}\sum_{i:s_{i}=B}\Phi(u_{i},v_{i})x_{u}x_{v}}
\end{align*}
Call $\mathcal{U}_{B}=B^{k}$. Then we can write every $\mathcal{U}=\{(s,u):u\in\mathcal{U}_{B}\}$
where clearly $\Phi((s,u),(s,v))=\sum_{i:s_{i}=B}\Phi(u_{i},v_{i})$.
We can apply the binomial theorem to the sum over types $s\in S$
to show 
\[
\sum_{\mathcal{U}}\gamma(\mathcal{U})=2^{3^{k}}\sum_{x\in y+\ker M_{k}}\prod_{j=1}^{k}\left(1+(-1)^{\sum_{\{u,v\}\subset\mathcal{U}_{B}}\Phi(u_{j},v_{j})x_{u}x_{v}}\right)
\]
Notice we have 
\begin{align*}
\sum_{\{u,v\}\subset\mathcal{U}_{B}}\Phi(u_{1},v_{1})x_{u}x_{v} & =\sum_{u_{1}=0,v_{1}=2}x_{u}x_{v}+\sum_{u_{1}=0,v_{1}=4}x_{u}x_{v}+\sum_{u_{1}=2,v_{1}=4}x_{u}x_{v}\\
 & =\sum_{u_{1}=0}x_{u}\sum_{v_{1}=2}x_{v}+\sum_{u_{1}=0}x_{u}\sum_{v_{1}=4}x_{v}+\sum_{u_{1}=2}x_{u}\sum_{v_{1}=4}x_{v}
\end{align*}
Which are both $0$ for any $x\in W+\ker M_{k}$, as the only possible
nonzero terms $\sum_{u_{1}=0}x_{u}$ are multiplied by another term
going to $0$. By symmetry the same is true for all $j=1,...,k$,
so we now have 
\begin{align*}
\sum_{\mathcal{U}}\gamma(\mathcal{U}) & =2^{3^{k}}\sum_{x\in y+\ker M_{k}}2^{k}\\
 & =2^{\dim\ker M_{k}+3^{k}+k}
\end{align*}
Recall that $\dim M_{k}=3^{k}-2k-1$. Also it follows that $W$ has
a basis of vectors $w$ such that $\sum_{u:u_{i}=j}w_{u}=1$ for exactly
one pair of $1\le i\le k$ and $j\in\{0,3\}$ (corresponding to the
type). In other words, $\dim W=k$. So we have 
\begin{align*}
\sum_{\mathcal{U}}\gamma(\mathcal{U}) & =2^{2\cdot3^{k}-k-1}
\end{align*}
Thus implying 
\begin{align*}
S_{k,0,0}(X) & =\frac{1}{2^{6k+1}}\left(\sum_{4n<X}\mu^{2}(2n)(3a)^{k\omega(n)}\right)+O\left(X(\log X)^{(3a)^{k}-a^{k}-1+\epsilon}\right)
\end{align*}
from which the corresponding case of \ref{thm:mainthm3} follows by
noting 
\begin{align*}
\sum_{4n<X}\mu^{2}(2n)(3a)^{k\omega(n)} & =\frac{2}{(3a)^{k}}\sum_{d\in\mathcal{D}_{X,4,8}^{+}}(3a)^{k\omega(d)}
\end{align*}

\subsection{The case $d>0$ and $d\equiv0$ mod $8$}

Next we consider fundamental discriminants $d>0$ and $d\equiv0$
mod $8$. The number of $H_{8}$ extensions of a quadratic field $k$
Galois over $\Q$ with such a discriminant is 
\[
f\left(d\right)=f_{1}\left(d\right)-f_{2}\left(d\right)+2^{\omega(d)-4}
\]
where we define

\begin{eqnarray*}
f_{1}\left(d\right) & = & \frac{1}{2}2^{-4}\sum_{d=8d_{1}d_{2}d_{3}}\left(1+\Leg{d_{2}d_{3}}{2}\right)\prod_{p\mid d_{3}}\left(1+\left(\frac{2d_{1}d_{2}}{p}\right)\right)\\
 &  & \times\prod_{p\mid d_{1}}\left(1+\left(\frac{d_{2}d_{3}}{p}\right)\right)\prod_{p\mid d_{2}}\left(1+\left(\frac{2d_{3}d_{1}}{p}\right)\right)
\end{eqnarray*}
and 
\begin{eqnarray*}
f_{2}\left(d\right) & = & 2^{-4}\sum_{d=8d_{1}d_{2}}\left(1+\Leg{d_{2}}{2}\right)\prod_{p\mid d_{1}}\left(1+\left(\frac{d_{2}}{p}\right)\right)\prod_{p\mid d_{2}}\left(1+\left(\frac{2d_{1}}{p}\right)\right)
\end{eqnarray*}
where the factoriation $d=d_{1}d_{2}d_{3}$ is into integers such
that $d_{1}\equiv-1\mod4$ and $d_{i}\equiv1\mod4$ otherwise by the
same reasoning as in the previous subsections, where $\alpha(d)=0$
is all odd primes dividing $d$ satisfy $p\equiv1mod4$ and $\alpha(d)=1$
otherwise. As before this gives 
\begin{eqnarray*}
f_{1}\left(d\right) & = & \frac{1}{2^{5}}\sum_{d=D_{0}D_{1}D_{2}D_{3}D_{4}D_{5}}\left(1+\Leg{D_{1}D_{2}D_{4}D_{5}}{2}\right)\Leg{2}{D_{2}D_{4}}\\
 &  & \times\left(\frac{D_{0}D_{3}D_{2}D_{5}}{D_{4}}\right)\left(\frac{D_{2}D_{5}D_{4}D_{1}}{D_{0}}\right)\left(\frac{D_{4}D_{1}D_{0}D_{3}}{D_{2}}\right)
\end{eqnarray*}
and 
\begin{eqnarray*}
f_{2}\left(d\right)=\frac{1}{2^{4}}\sum_{d=E_{0}E_{1}E_{2}E_{3}}\left(1+\Leg{E_{2}E_{3}}{2}\right)\Leg{2}{E_{2}}\left(\frac{E_{2}E_{3}}{E_{0}}\right)\left(\frac{E_{0}E_{1}}{E_{2}}\right)
\end{eqnarray*}
where the sum is over factorizations which satisfy the congruences
\begin{eqnarray*}
D_{2}D_{5},D_{1}D_{4} & \equiv & 1\mod4
\end{eqnarray*}
\begin{eqnarray*}
E_{2}E_{3} & \equiv & 1\mod4.
\end{eqnarray*}

Now define $Q_{B}\left(u\right)$ and $S_{C}\left(u\right)$ for $B\subset\left\{ 1,\ldots,j_{1}\right\} $
and $C\subset\{j_{1}+1,...,j_{2}\}$as in Section \ref{subsec:positive4mod8}.
As before we compute 
\begin{align*}
f_{j_{1},j_{2}}\left(d\right) & =\frac{1}{2^{5j_{1}+4j_{2}}}\sum_{\left(D_{u}\right)}\sum_{B\subset\{1,...,j_{1}\}}\sum_{C\subset\{j_{1}+1,...,j_{2}\}}\prod_{u}\Leg{D_{u}}{2}^{Q_{B}\left(u\right)+S_{C}\left(u\right)}\\
 & \times\Leg{2}{D_{u}}^{\lambda_{j_{1}}(u)+\gamma_{j_{2}}(u)}\prod_{u,v}\left(\frac{D_{u}}{D_{v}}\right)^{\Phi_{j_{1}}\left(u,v\right)+\Psi_{j_{2}}\left(u,v\right)}\\
\end{align*}
The sum is over $6^{j_{1}}4^{j_{2}}$ tuples of integers $\left(D_{u}\right)$
which satisfy $\prod_{u}D_{u}=d$ and the congruence conditions: for
all $1\le i\le j_{1}$ and $\left(u_{i},v_{i}\right)\in\left\{ \left(2,5\right),\left(4,1\right)\right\} $
and all $j_{1}+1\le i\le j_{1}+j_{2}$ and $\left(u_{i},v_{i}\right)\in\left\{ \left(2,3\right)\right\} $

\begin{eqnarray}
\prod_{u}D_{u}\prod_{v}D_{v} & \equiv & 1\mod4\label{eq:positive0mod8congruence}
\end{eqnarray}
where the above products are over all $u$ with $u_{i}$ in the $i$th
position and all $v$ with $v_{i}$ in the $i$th position.

Thus summing over discriminants $d<0$ with $d\equiv1\mod4$ we get
\begin{eqnarray*}
\sum_{d<X}2^{j_{3}\omega(d)-j_{3}}a^{\omega\left(d\right)}f_{j_{1},j_{2}}\left(d\right) & = & \frac{1}{3^{j_{1}}2^{5j_{1}+4j_{2}+k}}\sum_{\left(D_{u}\right)}\mu^{2}\left(\prod_{u}D_{u}\right)a^{k\omega\left(d\right)}2^{j_{3}\omega(d)}\\
 &  & \times\sum_{B\subset\{1,...,j_{1}\}}\sum_{C\subset\{j_{1}+1,...,j_{2}\}}\prod_{u}\Leg{D_{u}}{2}^{Q_{B}\left(u\right)+S_{C}\left(u\right)}\\
 &  & \times\Leg{2}{D_{u}}^{\lambda_{j_{1}}(u)+\gamma_{j_{2}}(u)}\prod_{u,v}\left(\frac{D_{u}}{D_{v}}\right)^{\Phi_{j_{1}}\left(u,v\right)+\Psi_{j_{2}}\left(u,v\right)}
\end{eqnarray*}
where the sum is over $6^{j_{1}}4^{j_{2}}$ tuples of integers $\left(D_{u}\right)$
which satisfy $\prod_{u}D_{u}<X$ and the conditions $\left(\ref{eq:positive0mod8congruence}\right)$.

In this case, as in the $d>0$ and $d\equiv1\mod4$ case, we have
\begin{eqnarray*}
\sum_{\mathbf{A}\mathrm{\ admissible\ for\ }\mbox{\ensuremath{\mathcal{U}}}}S_{k,0,0}\left(X,\mathbf{A}\right) & = & \frac{1}{2^{5k}}\left(\sum_{\left(D_{u}\right)}2^{-k}\mu^{2}\left(d/2\right)a^{k\omega\left(d/4\right)}\right)\\
 &  & \times\sum_{\left(h_{u}\right)}\left[\sum_{C\subset\{1,...,k\}}\prod_{u}\left(-1\right)^{\sum_{i}Q(u_{i})\chi_{C}(i)\frac{h_{u}^{2}-1}{8}+\lambda_{k}(u)\frac{h_{u}^{2}-1}{8}}\right]\\
 &  & \times\left[\prod_{u,v}\left(-1\right)^{\Phi_{k}\left(u,v\right)\frac{h_{u}-1}{2}\frac{h_{v}-1}{2}}\right]+O\left(X\left(\log X\right)^{(3a)^{k}-a^{k}-1+\epsilon}\right).
\end{eqnarray*}

For $h_{u}$ the congruence class of $D_{u}\mod8$ and $d=8\prod D_{u}$
and we grouped the $8$ factor with the first discriminant in the
factorization (i.e. for $k=1$ the factorization is $d=(8D_{0}D_{3})(D_{1}D_{4})(D_{2}D_{5})$).

Noting that $h_{u}$ is one out of four choices for odd numbers $\mod8$,
we get 
\begin{eqnarray*}
S_{k,0,0}(X) & = & \frac{1}{2^{6k}2^{2\cdot3^{k}}}\left(\sum_{\mathcal{U}}\gamma(\mathcal{U})\right)\left(\sum_{8n<X}\mu^{2}(2n)(3a)^{k\omega(n)}\right)+O\left(X(\log X)^{(3a)^{k}-a^{k}-1+\epsilon}\right)
\end{eqnarray*}
where we define 
\begin{eqnarray*}
\gamma(\mathcal{U}) & = & \sum_{\left(h_{u}\right)}\left[\sum_{C\subset\{1,...,k\}}\prod_{u}\left(-1\right)^{\sum_{i}Q(u_{i})\chi_{C}(i)\frac{h_{u}^{2}-1}{8}}\right]\left[\prod_{u}(-1)^{\lambda_{k}(u)\frac{h_{u}^{2}-1}{8}}\right]\left[\prod_{u,v}\left(-1\right)^{\Phi_{k}\left(u,v\right)\frac{h_{u}-1}{2}\frac{h_{v}-1}{2}}\right]
\end{eqnarray*}
allowing odd congruence classes $h_{u}\mod8$ satisfying the following
conditions: for all $1\le i\le k$ 
\begin{eqnarray}
\prod_{u}h_{u}\prod_{v}h_{v}\equiv\begin{cases}
1\mod4 & (u_{i},v_{i})\in\{(1,4),(2,5)\}\\
\prod_{u'}h_{u'}\prod_{v'}h_{v'} & (u_{i},v_{i})=(u_{i}',v_{i}')=(0,3)
\end{cases}
\end{eqnarray}
where the above products are taken over all $u$ with $u_{i}$ in
the $i$th position and $v$ with $v_{i}$ in the $i$th position.

If we call $x\equiv\left(\frac{h_{u}-1}{2}\right)\mod2$ an element
of $\F_{2}^{3^{k}}$, then $x\in y+\ker M_{k}$ one of two cosets
of $\ker M_{k}$ depending on the congruence class of $\prod_{u}h_{u}\prod_{v}h_{v}$
for $(u_{i},v_{i})=(0,3)$.

Now consider that $\Phi(u,v)=0$ if $u,v\in A=\{1,3,5\}$. $\mathcal{U}$
is a largest maximal unlinked set, and so has a type $s\in S$, so
it follows that 
\[
\Phi_{k}(u,v)=\sum_{i}\Phi(u_{i},v_{i})=\sum_{i:s_{i}=B}\Phi(u_{i},v_{i})
\]
This way we show that 
\begin{eqnarray*}
\sum_{\mathcal{U}}\gamma(\mathcal{U}) & = & \sum_{s\in S}\sum_{y}\sum_{x\in y+\ker M_{k}}\left[\sum_{(h_{u}):\frac{h_{u}-1}{2}\equiv x_{u}\mod2}\sum_{C\subset\{1,...,k\}}\prod_{u}\left(-1\right)^{\sum_{i\in C}Q(u_{i})\left(\frac{h_{u}^{2}-1}{8}\right)}\right]\\
 &  & \left[\times\prod_{u}(-1)^{\sum_{s:s_{i}=B}\lambda(u_{i})\frac{h_{u}^{2}-1}{8}}\right]\\
 &  & \times\left[\prod_{\{u,v\}}\left(-1\right)^{\sum_{s_{i}=B}\Phi(u_{i},v_{i})x_{u}x_{v}}\right]
\end{eqnarray*}
Where the sum of $y$ is over $2^{k}$ cosets of $\ker M_{k}$ to
account for the missing condition on $u_{i}\in\{0,3\}$. Notice that
for each $x_{u}$, there are two choices of $h_{u}\mod8$ such that
$\frac{h_{u}-1}{2}\equiv x_{u}\mod2$, because $h_{u}$ and $5h_{u}$
give the same image. If we fix an $(h(x)_{u})$ satisfying this property
without loss of generality also satisfying $\frac{h_{u}^{2}-1}{8}\equiv0\mod2$,
then we have 
\begin{align*}
 & \sum_{(h_{u}):\frac{h_{u}-1}{2}\equiv x_{u}\mod2}\sum_{C\subset\{1,...,k\}}\prod_{u}\left(-1\right)^{\left(\sum_{i\in C}Q(u_{i})+\lambda_{k}(u)\right)\left(\frac{h_{u}^{2}-1}{8}\right)}\\
 & =\sum_{T\subset\mathcal{U}}\sum_{C\subset\{1,...,k\}}\prod_{u}\left(-1\right)^{\left(\sum_{i\in C}Q(u_{i})+\lambda_{k}(u)\right)\left(\frac{(h_{u}5^{\chi_{T}(u)})^{2}-1}{8}\right)}\\
 & =\sum_{T\subset\mathcal{U}}\sum_{C\subset\{1,...,k\}}\prod_{u}\left(-1\right)^{\left(\sum_{i\in C}Q(u_{i})+\lambda_{k}(u)\right)\left(\frac{5^{2\chi_{T}(u)}-1}{8}\right)}
\end{align*}
Now, $\frac{5^{2}-1}{8}\equiv1\mod2$. So it follows that $\left(\frac{5^{2\chi_{T}(u)}-1}{8}\right)\equiv(\chi_{T}(u))\mod2$
so we have 
\begin{align*}
 & \sum_{(h_{u}):\frac{h_{u}-1}{2}\equiv x_{u}\mod2}\sum_{C\subset\{1,...,k\}}\prod_{u}\left(-1\right)^{\left(\sum_{i\in C}Q(u_{i})+\lambda_{k}(u)\right)\left(\frac{h_{u}^{2}-1}{8}\right)}\\
 & =\sum_{T\subset\mathcal{U}}\sum_{C\subset\{1,...,k\}}\prod_{u}\left(-1\right)^{\left(\sum_{i\in C}Q(u_{i})+\lambda_{k}(u)\right)\chi_{T}(u)}\\
 & =\sum_{T\subset\mathcal{U}}\sum_{C\subset\{1,...,k\}}\left(-1\right)^{\sum_{u,i}Q(u_{i})\chi_{C}(i)\chi_{T}(u)+\lambda_{k}(u)\chi_{T}(u)}\\
 & =\sum_{C\subset\{1,...,k\}}\prod_{u}\left(1+(-1)^{\sum_{i}Q(u_{i})\chi_{C}(i)+\lambda(u_{i})}\right)
\end{align*}
We then have 
\begin{align*}
Q(u_{i})\chi_{C}(i)+\lambda(u_{i}) & =\begin{cases}
0 & u_{i}\in\{0,3,2,4\},i\in C\\
1 & u_{i}\in\{1,5\},i\in C\\
0 & u_{i}\in\{0,3,1,5\},i\not\in C\\
1 & u_{i}\in\{2,4\},i\not\in C
\end{cases}
\end{align*}
Suppose that $C\not\subset\{i:s_{i}=B\}$, then fix some $j\in C$
such that $s_{j}=A$. Choose $u\in\mathcal{U}$ such that 
\begin{align*}
u_{i}\in\begin{cases}
\{1\} & i=j\\
\{0,3\} & i\ne j
\end{cases}
\end{align*}
Then we get the summands are zero corresponding to these $C$. Suppose
next that $\{i:s_{i}=B\}\not\subset C$, then fix $j\not\in C$ such
that $s_{j}=B$ and choose $u\in\mathcal{U}$ such that 
\begin{align*}
u_{i}\in\begin{cases}
\{2\} & i=j\\
\{0,3\} & i\ne j
\end{cases}
\end{align*}
Then these summands give us zero as well. The only summand remaining
is $C=\{i:s_{i}=B\}$, which clearly gives $2^{3^{k}}$.

So then we have 
\begin{align*}
\sum_{\mathcal{U}}\gamma(\mathcal{U}) & =2^{3^{k}}\sum_{s\in S}\sum_{y}\sum_{x\in y+\ker M_{k}}\left[\prod_{\{u,v\}}\left(-1\right)^{\sum_{s_{i}=B}\Phi(u_{i},v_{i})x_{u}x_{v}}\right]\\
 & =2^{3^{k}}\sum_{y}\sum_{s\in S}\sum_{x\in y+\ker M_{k}}(-1)^{\sum_{\{u,v\}}\sum_{i:s_{i}=B}\Phi(u_{i},v_{i})x_{u}x_{v}}
\end{align*}
Noting that there are $2$ choices for $y$, the same proof for the
case $d>0$, $d\equiv4\mod8$ implies 
\begin{align*}
S_{k,0,0}(X) & =\frac{1}{2^{6k}}\left(\sum_{8n<X}\mu^{2}(2n)(3a)^{k\omega(n)}\right)+O\left(X(\log X)^{(3a)^{k}-a^{k}-1+\epsilon}\right)
\end{align*}
Then the corresponding case of \ref{thm:mainthm3} follows from 
\begin{align*}
\sum_{8n<X}\mu^{2}(2n)(3a)^{k\omega(n)} & =\frac{1}{(3a)^{k}}\sum_{d\in\mathcal{D}_{X,0,8}^{+}}(3a)^{k\omega(d)}+o(X).
\end{align*}

 \bibliographystyle{plain}
\bibliography{bibliography}

\end{document}